\documentclass[oneside,10pt]{article}
\usepackage[b5paper]{geometry}	

\usepackage{amsfonts,amsmath,latexsym,amssymb}
\usepackage{mathrsfs,upref}
\usepackage{mathptmx}		    	                            		
\usepackage{amscd}
\usepackage{amsthm}
\usepackage{graphicx}
\usepackage{amsmath}
\usepackage{lipsum}

\newcommand\blfootnote[1]{%
  \begingroup
  \renewcommand\thefootnote{}\footnote{#1}%
  \addtocounter{footnote}{-1}%
  \endgroup}

\newtheorem{theorem}{Theorem}

\newtheorem{corollary}[theorem]{Corollary}

\newtheorem{definition}[theorem]{Definition}

\newtheorem{lemma}[theorem]{Lemma}

\newtheorem{proposition}[theorem]{Proposition}
\newtheorem{remark}[theorem]{Remark}

\begin{document}

\title{On potentials of distributions in Orlicz-Hardy type spaces on the Heisenberg group}
\author{Pablo Rocha}

\maketitle

\begin{abstract}
In this work, we introduce Orlicz-Hardy type spaces and Orlicz-Calder\'on Hardy type spaces on the Heisenberg group $\mathbb{H}^{n}$ and study the relationship between them by means of the Heisenberg sub-Laplacian $\mathcal{L}$. More precisely, we show, under suitable assumptions, that every distribution in the Orlicz-Hardy space $H^{\Phi}(\mathbb{H}^{n})$ can be represented uniquely as the sub-Laplacian of a function in an appropriate Orlicz-Calder\'on Hardy space. In this way, for any $f \in H^{\Phi}(\mathbb{H}^{n})$,  we obtain a uniqueness and solvability result for the equation $\mathcal{L}F=f$.
\end{abstract}


\section{Introduction}

On the Heisenberg group $\mathbb{H}^n$, we consider the following inhomogeneous equation
\begin{equation} \label{L problem}
\mathcal{L} F = f,
\end{equation}
where $\mathcal{L}$ is the Heisenberg sub-Laplacian, $f$ is a data distribution on $\mathbb{H}^n$ and $F$ is an unknown function. Recently, the author in \cite{rocha3} proved that if $f \in H^p(\mathbb{H}^n)$ with $Q \, (2 + \frac{Q}{q})^{-1} < p \leq 1$ (where $Q = 2n+2$ and 
$1 < q < \frac{n+1}{n}$), then there exists a unique $F \in \mathcal{H}^{p}_{q, 2}(\mathbb{H}^n)$ what solves (\ref{L problem}). Here, 
$\mathcal{H}^{p}_{q, 2}(\mathbb{H}^n)$ is the Calder\'on-Hardy space on $\mathbb{H}^n$. This problem was also posed in the variable context. In \cite{rocha4}, we solved the equation (\ref{L problem}) for $f \in H^{p(\cdot)}(\mathbb{H}^n)$, and certain variable exponents 
$p(\cdot) : \mathbb{H}^n \to (0, \infty)$. In this work, for an Orlicz function $\Phi$, we introduce the Orlicz-Hardy spaces on the Heisenberg group $H^{\Phi}(\mathbb{H}^n)$ and posed the equation (\ref{L problem}) for $f \in H^{\Phi}(\mathbb{H}^n)$. Naturally, the classical Hardy spaces $H^{p}(\mathbb{H}^n)$, $0 < p < \infty$, are variable Hardy spaces and Orlicz-Hardy spaces. These last two spaces have different nature. That is, $H^{p(\cdot)}(\mathbb{H}^n)$ mixes the value of $p$ according to the position of $z$ and $H^{\Phi}(\mathbb{H}^n)$ mixes the value of $p$ according to the value of $f(z)$. 

The counterpart of the equation (\ref{L problem}) on $\mathbb{R}^n$ was first considered by A. Gatto, J. Jim\'enez and C. Segovia in \cite{segovia}. More precisely, they posed, for $0< p \leq 1$, $m \in \mathbb{N}$ and $f \in H^p(\mathbb{R}^n)$ (see \cite{fefferman}), the equation
\begin{equation} \label{delta problem}
\Delta^m F = f,
\end{equation}
where $\Delta$ is the Laplace operator on $\mathbb{R}^n$. To address this problem, they introduced the Calder\'on-Hardy spaces
$\mathcal{H}^p_{q, \gamma}(\mathbb{R}^n)$, $0 < p \leq 1 < q < \infty$ and $\gamma > 0$, and proved for $n(2m + n/q)^{-1} < p \leq 1$
that given $f \in H^p(\mathbb{R}^n)$ there exists a unique $F \in \mathcal{H}^p_{q, 2m}(\mathbb{R}^n)$ what solves (\ref{delta problem}).

In \cite{Duran}, R. Dur\'an extended the definition of $\mathcal{H}^p_{q, 2m}(\mathbb{R}^n)$ to the case of non-isotropic dilations on 
$\mathbb{R}^n$, solving the problem (\ref{delta problem}) for more general elliptic operators with symbols of the form 
$\xi_1^{2k_1} + \cdot \cdot \cdot + \xi_n^{2k_n}$, with $k_1, ..., k_n \in \mathbb{N}$.

The equation (\ref{delta problem}), for $f \in H^{p(\cdot)}(\mathbb{R}^{n})$ and $f \in H^{p}(\mathbb{R}^{n}, w)$, was studied by the present author in \cite{rocha1} and \cite{rocha2} respectively, obtaining analogous results to those of Gatto, Jim\'enez and Segovia.

Recently, Z. Liu, Z. He and H. Mo in \cite{He} extended the definition of Calder\'on-Hardy spaces to Orlicz setting on $\mathbb{R}^n$ and  solved the equation (\ref{delta problem}) when $f \in H^{\Phi}(\mathbb{R}^{n})$, where $H^{\Phi}(\mathbb{R}^{n})$ are the Orlicz-Hardy spaces defined in \cite{Nakai}.

The purpose of this work is to solve the equation (\ref{L problem}) for $f \in H^{\Phi}(\mathbb{H}^n)$ and $Q \, (2 + \frac{Q}{q})^{-1} < 
i(\Phi) \leq I(\Phi) < \infty$, where $Q= 2n+2$, $1 < q < \frac{n+1}{n}$ and the index $i(\Phi)$ and $I(\Phi)$ are given by (\ref{p menos}) and (\ref{p mas}) below. Once defined the spaces $H^{\Phi}(\mathbb{H}^n)$ and $\mathcal{H}^{\Phi}_{q, 2}(\mathbb{H}^{n})$, we will prove 
that $f$ has an atomic decomposition $f = \sum_j \lambda_j a_j$ in $\mathcal{S}'(\mathbb{H}^n)$. Then, as in \cite{rocha3}, we consider the following potential of $f$,
\[ 
f \ast c_n \rho^{-2n} = \sum_j \lambda_j (a_j \ast c_n \rho^{-2n}),
\]
where "$\ast$" is the convolution in $\mathbb{H}^n$ and $c_n \rho^{-2n}$ is the fundamental solution of $\mathcal{L}$ obtained by G. Folland in \cite{Folland}, and show that a representative for the solution $F \in \mathcal{H}^{\Phi}_{q, 2}(\mathbb{H}^{n})$ of (\ref{L problem}) is 
$\sum_j \lambda_j (a_j \ast_{\mathbb{H}^n} c_n \rho^{-2n})$. On the other hand, the case $0 < I(\Phi) < Q \, (2 + \frac{Q}{q})^{-1}$ 
is trivial. Indeed, we have $\mathcal{H}^{\Phi}_{q, 2}(\mathbb{H}^{n}) = \{ 0\}$, when $0 < I(\Phi) < Q \, (2 + \frac{Q}{q})^{-1}$. 
These results is contained in Theorems \ref{principal thm} and \ref{2nd thm} of Section \ref{main thm}, respectively. 

Although the fundamental solutions for the iterated Heisenberg sub-Laplacian $\mathcal{L}^m$ are known for every integer $m \geq 2$ 
(see \cite{Benson}), the problem $\mathcal{L}^m F = f$ on $\mathbb{H}^n$ is much more complicated. For this reason we focus solely on 
the case $m=1$.

The tools used in this work are the atomic series for elements in $H^{\Phi}(\mathbb{H}^n)$ and the fundamental solution of $\mathcal{L}$ already mentioned, together with the complementary function $\Phi^{*}$, a Fefferman-Stein inequality for the Hardy-Littlewood maximal 
operator, and certain pointwise inequalities established in Section \ref{OCH}, among them, (\ref{N estimate}) is the most important.

Our results apply to the following Orlicz functions:

\

(i) $\Phi(t) = t^p$, with $t \geq 0$ and $0 < p < \infty$. It this case, $L^{\Phi}(\mathbb{H}^n) \equiv L^{p}(\mathbb{H}^n)$.

(ii) $\Phi(t) = t^{p_1} + t^{p_2}$, with $t \geq 0$ and $0 < p_1 \leq p_2 < \infty$. In this case, $L^{\Phi}(\mathbb{H}^n)$ is isomorphic to 
$L^{p_1}(\mathbb{H}^n) \cap L^{p_2}(\mathbb{H}^n)$.

(iii) $\Phi(t) = \min \{ t^{p_1}, t^{p_2} \}$, with $t \geq 0$ and $0 < p_1 \leq p_2 < \infty$. In this case, $L^{\Phi}(\mathbb{H}^n)$ is isomorphic to $L^{p_1}(\mathbb{H}^n) + L^{p_2}(\mathbb{H}^n)$.

(iv) $\Phi(t) = t \log(e + t)$, with $t \geq 0$. In this case, $L^{\Phi}(\mathbb{H}^n)$ is isomorphic to $L \log L (\mathbb{H}^n)$.

\

We observe that the functions (i)-(iv) are all of positive lower and upper type (see Definition \ref{positive type} below). 

\

This paper is organized as follows. In Section \ref{Heisn} we state the basics of the Heisenberg group. The properties of Orlicz functions and Orlicz spaces are presented in Section \ref{LO}. The definition of Orlicz-Hardy spaces on the Heisenberg group as well as its atomic decomposition are presented in Section \ref{OH}. We introduce the Orlicz-Calder\'on Hardy spaces on the Heisenberg group and investigate their properties in Section \ref{OCH}. Finally, our main results are proved in Section \ref{main thm}.

\

\textbf{Notation:} The symbol $S \lesssim T$ stands for the inequality $S \leq c T$ for some constant $c$. The symbol $S \approx T$ 
stands for $T \lesssim S \lesssim T$. For a measurable subset $E\subseteq \mathbb{H}^{n}$ we denote by $\left\vert E\right\vert $ and 
$\chi_{E}$ the Haar measure of $E$ and the characteristic function of $E$ respectively. Given a real number $s \geq 0$, we write 
$\lfloor s \rfloor$ for the integer part of $s$.

Throughout this paper, $C$ will denote a positive constant, not necessarily the same at each occurrence.

\section{The Heisenberg group} \label{Heisn}

The  Heisenberg group $\mathbb{H}^{n}$ can be identified with $\mathbb{R}^{2n} \times \mathbb{R}$ whose group law 
(noncommutative) is given by
\[
(x,t) \cdot (y,s) = \left( x+y, t+s + x^{t} J y \right),
\]
where $J$ is the $2n \times 2n$ skew-symmetric matrix given by
\[
J= 2 \left( \begin{array}{cc}
                           0 & -I_n \\
                           I_n & 0 \\
                                      \end{array} \right)
\]
being $I_n$ the $n \times n$ identity matrix. One can easily check that $e = (0,0)$ is the neutral element, and $(x, t)^{-1}=(-x, -t)$ is the inverse of $(x, t)$.

The Heisenberg group $\mathbb{H}^n$ also admits the following homogeneous dilation: 
\[
\lambda \cdot (x,t) = (\lambda x, \lambda^{2}t), \,\,\,\ \lambda > 0,
\]
which satisfies $\lambda \cdot((x,t) \cdot (y,s)) = (\lambda \cdot(x,t)) \cdot (\lambda \cdot(y,s))$. 

The {\it Koranyi norm} on $\mathbb{H}^{n}$ is the function $\rho : \mathbb{H}^{n} \to [0, \infty)$ defined by
\begin{equation} \label{Koranyi norm}
\rho(z) = \rho(x,t) = \left( |x|^{4} + \, t^{2}  \right)^{1/4}, \,\,\,\, z=(x,t) \in \mathbb{H}^{n},
\end{equation}
where $| \cdot |$ is the usual Euclidean norm on $\mathbb{R}^{2n}$. It is easy to check that for any $z, w \in \mathbb{H}^n$ and $\lambda > 0$, 

\

(i) $\rho(\lambda \cdot z) = \lambda \rho(z)$,

(ii) $\rho(z + w) \leq \rho(z) + \rho(w)$,

(iii) $|x| \leq \rho(x,t)$ and $|t| \leq \rho(x,t)^2$.
\\\\
Moreover, $\rho$ is continuous on $\mathbb{H}^{n}$ and is smooth on $\mathbb{H}^{n} \setminus \{ e \}$. The $\rho$ - ball centered at 
$z_0 \in \mathbb{H}^{n}$ with radius $\delta > 0$ is defined by
\[
B(z_0, \delta) := \{ w \in \mathbb{H}^{n} : \rho(z_0^{-1} \cdot w) < \delta \}.
\]

The topology in $\mathbb{H}^{n}$ induced by the $\rho$ - balls coincides with the Euclidean topology of 
$\mathbb{R}^{2n} \times \mathbb{R} \equiv\mathbb{R}^{2n+1}$ (see \cite[Proposition 3.1.37]{Fischer}). So, the borelian sets of 
$\mathbb{H}^{n}$ are identified with those of $\mathbb{R}^{2n+1}$. The Haar measure in $\mathbb{H}^{n}$ is the Lebesgue measure of 
$\mathbb{R}^{2n+1}$, thus $L^{p}(\mathbb{H}^{n}) \equiv L^{p}(\mathbb{R}^{2n+1})$, for every $0 < p \leq \infty$. Moreover, for 
$f \in L^{1}(\mathbb{H}^{n})$ and for $r > 0$ fixed, we have
\begin{equation} \label{homog dim}
\int_{\mathbb{H}^{n}} f(r \cdot z) \, dz = r^{-Q} \int_{\mathbb{H}^{n}} f(z) \, dz,
\end{equation}
where $Q= 2n+2$. The number $2n+2$ is known as the {\it homogeneous dimension} of $\mathbb{H}^{n}$ (we observe that the {\it topological dimension} of $\mathbb{H}^{n}$ is $2n+1$).

Let $|B(z_0, \delta)|$ be the Haar measure of the $\rho$ - ball $B(z_0, \delta) \subset \mathbb{H}^{n}$. Then, 
\[
|B(z_0, \delta)| = c \delta^{Q},
\]
where $c = |B(e,1)|$ and  $Q = 2n+2$. Given $\lambda > 0$, we put $\lambda B = \lambda B(z_0, \delta) = 
B(z_0, \lambda \delta)$. So $|\lambda B| = \lambda^{Q}	|B|$.

\begin{remark}
For any $z, z_0 \in \mathbb{H}^{n}$ and $\delta >0$, we have
\[
z_0 \cdot B(z, \delta) = B(z_0 \cdot z, \delta).
\]
In particular, $B(z, \delta) = z \cdot B(e, \delta)$. It is also easy to check that $B(e, \delta) = \delta \cdot B(e,1)$ for any $\delta > 0$.
\end{remark}
\begin{remark} \label{cambio de centro}
If $f \in L^{1}(\mathbb{H}^{n})$, then for every $\rho$ - ball $B$ and every $z_0 \in \mathbb{H}^{n}$, we have
\[
\int_{B} f(w) \, dw = \int_{z_{0}^{-1} \cdot B} f(z_0 \cdot u) \, du.
\]
\end{remark}

\begin{definition} \label{radial}
A function $f : \mathbb{H}^n \to \mathbb{C}$ is said to be radial if there exists a function $f_0 : [0, \infty) \to \mathbb{C}$ such that
\[
f(x,t) = f_0(\rho(x,t)), \,\,\,\, \text{for all} \,\, (x,t) \in \mathbb{H}^n. 
\]
\end{definition}

\begin{remark}
We observe that the literature is not unanimous in the use of this terminology. Sometimes, a function $f : \mathbb{H}^n \to \mathbb{C}$ is said to be radial if there exists $f_0 : [0, \infty) \times \mathbb{R} \to \mathbb{C}$ such that $f(x,t) = f_0 (|x|, t)$.
\end{remark}

The Hardy-Littlewood maximal operator $M$ is defined by
\[
Mf(z) = \sup_{B \ni z} |B|^{-1}\int_{B} |f(w)| \, dw,
\]
where $f$ is a locally integrable function on $\mathbb{H}^{n}$ and the supremum is taken over all the $\rho$ - balls $B$ containing $z$.

\

If $f$ and $g$ are measurable functions on $\mathbb{H}^{n}$, their convolution $f * g$ is defined by
\[
(f * g)(z) := \int_{\mathbb{H}^{n}} f(w) g(w^{-1} \cdot z) \, dw,
\]
when the integral is finite.

For every $i = 1,2, ..., 2n+1$, $X_i$ denotes the left invariant vector field given by
\[
X_i = \frac{\partial}{\partial x_i} + 2 x_{i+n} \frac{\partial}{\partial t}, \,\,\,\, i=1, 2, ..., n;
\]
\[
X_{i+n} = \frac{\partial}{\partial x_{i+n}} - 2 x_{i} \frac{\partial}{\partial t}, \,\,\, i=1, 2, ..., n;
\]
and
\[
X_{2n+1} = \frac{\partial}{\partial t}.
\]
The sublaplacian on $\mathbb{H}^n$, denoted by $\mathcal{L}$, is the counterpart of the Laplacain $\Delta$ on $\mathbb{R}^n$. The sublaplacian $\mathcal{L}$ is defined by
\[
\mathcal{L} =- \sum_{i=1}^{2n} X_i^2,
\]
where $X_i$, $i= 1, ..., 2n$, are the left invariant vector fields defined above.

Given a multi-index $I=(i_1,i_2, ..., i_{2n}, i_{2n+1}) \in (\mathbb{N} \cup \{ 0 \})^{2n+1}$, we set
\[
|I| = i_1 + i_2 + \cdot \cdot \cdot + i_{2n} + i_{2n+1}, \hspace{.5cm} d(I) = i_1 + i_2 + \cdot \cdot \cdot + i_{2n} + 2 \, i_{2n+1}.
\]
The amount $|I|$ is called the length of $I$ and $d(I)$ the homogeneous degree of $I$. We adopt the following multi-index notation for
higher order derivatives and for monomials on $\mathbb{H}^{n}$. If $I=(i_1, i_2, ..., i_{2n+1})$ is a multi-index,  
$z = (x,t) = (x_1, ..., x_{2n}, t) \in \mathbb{H}^{n}$ and $X = \{ X_i \}_{i=1}^{2n+1}$, we put
\[
z^{I} := x_{1}^{i_1} \cdot \cdot \cdot x_{2n}^{i_{2n}} \cdot t^{i_{2n+1}} \,\,\,\,\,\, \text{and} \,\,\,\,\,\, 
X^{I} := X_{1}^{i_1} X_{2}^{i_2} \cdot \cdot \cdot X_{2n+1}^{i_{2n+1}}.
\]
A computation give
\[
(r\cdot z)^{I} = r^{d(I)} z^{I} \,\,\,\,\,\, \text{and}  \,\,\,\,\,\, X^{I}(f(r \cdot z)) = r^{d(I)} (X^{I}f)(r\cdot z).
\]
So, the operators $X^{I}$ and the monomials $z^{I}$ are homogeneous of degree $d(I)$. In particular, the 
sublaplacian $\mathcal{L}$ is an operator homogeneous of degree $2$. The operators $X^{I}$, and $\mathcal{L}$  
interact with the convolutions in the following way
\[
X^{I}(f \ast g) = f \ast (X^{I}g), \,\,\,\,\,\, \text{and} \,\,\,\,\,\, \mathcal{L} (f \ast g) = f \ast \mathcal{L}g.
\]

We recall that every polynomial $p$ on $\mathbb{H}^n$ can be written as a unique finite linear combination of the monomials $z^I$, that is
\begin{equation} \label{polynomial}
p(z) = \sum_{I \in \mathbb{N}_{0}^n} c_I \, z^I,
\end{equation}
where all but finitely many of the coefficients $c_I \in \mathbb{C}$ vanish. The \textit{homogeneous degree} of a polynomial $p$ written as 
(\ref{polynomial}) is $\max \{ d(I) : I \in \mathbb{N}_{0}^n \,\, \text{with} \,\, c_I \neq 0 \}$. Let $k \in \mathbb{N} \cup \{ 0 \}$, 
with $\mathcal{P}_{k}$ we denote the subspace formed by all the polynomials of homogeneous degree at 
most $k$. So, every $p \in \mathcal{P}_{k}$ can be written as $p(z) = \sum_{d(I) \leq k} c_I \, z^I$, with $c_I \in \mathbb{C}$.

The Schwartz space $\mathcal{S}(\mathbb{H}^{n})$ is defined by
\[
\mathcal{S}(\mathbb{H}^{n}) = \left\{ \phi \in C^{\infty}(\mathbb{H}^{n}) : \sup_{z \in \mathbb{H}^{n}} (1+\rho(z))^{N} 
|(X^{I} \phi)(z)| < \infty \,\,\, \forall \,\, N \in \mathbb{N}_{0}, \, I \in (\mathbb{N}_{0})^{2n+1}  \right\}.
\]
We topologize the space $\mathcal{S}(\mathbb{H}^{n})$ with the following family of seminorms
\[
\| \phi \|_{\mathcal{S}(\mathbb{H}^{n}), \, N} := \sum_{d(I) \leq N} \sup_{z \in \mathbb{H}^{n}} (1+\rho(z))^{N} |(X^{I} \phi)(z)| \,\,\,\,\,\,\, 
(N \in \mathbb{N}_{0}),
\]
with $\mathcal{S}'(\mathbb{H}^{n})$ we denote the dual space of $\mathcal{S}(\mathbb{H}^{n})$. 

A fundamental solution for the sublaplacian on $\mathbb{H}^n$ was obtained by G. Folland in \cite{Folland}. More precisely, he proved the following result.

\begin{theorem} \label{fund solution}
$c_n \, \rho^{-2n}$ is a fundamental solution for $\mathcal{L}$ with source at $0$, where
\[
\rho(x,t) = (|x|^4 + t^2)^{1/4},
\]
and
\[
c_n = \left[ n(n+2) \int_{\mathbb{H}^n} |x|^2 (\rho(x,t)^4 + 1)^{-(n+4)/2} dxdt \right]^{-1}.
\]
In others words, for any $u \in \mathcal{S}(\mathbb{H}^{n})$, $\left( \mathcal{L}u,  c_n \rho^{-2n} \right) = u(0)$.
\end{theorem}

\section{Orlicz spaces on the Heisenberg group} \label{LO}

We start this section with some basic notions and results about Orlicz functions and Orlicz spaces on $\mathbb{H}^n$ (see e.g. \cite{Hasto}).

\begin{definition}
A function $\Phi : [0, \infty) \to [0, \infty)$ is called an Orlicz function if

(i) it is non-decreasing and satisfies $\lim_{t \to 0^{+}} \Phi(t) = \Phi(0) = 0$, $\Phi(t) > 0$ for all $t>0$ and 
$\lim_{t \to \infty} \Phi(t) = \infty$;

(ii) the function $z \to \Phi(\vert f(z) \vert)$ is measurable for every measurable function $f$ on $\mathbb{H}^n$.
\end{definition}

\begin{definition}
Let $\Phi$ be an Orlicz function, the Orlicz space $L^{\Phi}(\mathbb{H}^n)$ is defined to be the set of all measurable functions $f$ on 
$\mathbb{H}^n$ such that
\[
\Vert f \Vert_{\Phi} := \inf \left\{ \lambda > 0 : \int_{\mathbb{H}^n} \Phi(\vert f(z) \vert/ \lambda) dz \leq 1 \right\} < \infty. 
\]
The amount $\Vert f \Vert_{\Phi}$ is known as the Luxemburg norm of $f$ with respect to $\Phi$.
\end{definition}

Given an Orlicz function $\Phi$ and any $s \in (0, \infty)$ fixed, we define
\[
\Phi_s(t) := \Phi(t^s), \,\,\,\,  t \in (0, \infty).
\]

\begin{lemma} \label{fits}
Let $\Phi$ be an Orlicz function. Then, for any $s \in (0, \infty)$ and any measurable function $f$, one has
\[
\Vert f \Vert_{\Phi}^{s} = \Vert \vert f \vert^s \Vert_{\Phi_{1/s}}.
\]
\end{lemma}

\begin{proof}
Fix $s \in (0, \infty)$. Then, by definition
\[
\Vert \vert f \vert^s \Vert_{\Phi_{1/s}} = \inf \left\{ \lambda > 0 : \int_{\mathbb{H}^n} \Phi(\vert f(z) \vert/ \lambda^{1/s}) dz \leq 1 \right\} 
\]
\[
= \inf \left\{ \mu^s > 0 : \int_{\mathbb{H}^n} \Phi(\vert f(z) \vert/ \mu) dz \leq 1 \right\} = \Vert f \Vert_{\Phi}^{s}.
\]
Then, the lemma follows.
\end{proof}

\begin{lemma} \label{fi menor uno}
Let $\Phi$ be a continuous Orlicz function. If $0 \neq f \in L^{\Phi}(\mathbb{H}^n)$, then
\[
\int_{\mathbb{H}^n} \Phi \left( \| f \|^{-1}_{\Phi} |f(z)|  \right) dz \leq 1.
\] 
\end{lemma}

\begin{proof}
From the definition of the amount $\| f \|_{\Phi}$, there exists a positive decreasing sequence $\{ \alpha_j \}$ such that 
$\int_{\mathbb{H}^n} \Phi \left( \alpha^{-1}_j |f(z)|  \right) dz \leq 1$ for all $j$ and $\alpha_j \to \| f \|_{\Phi} > 0$. 
Finally, by Fatou's lemma and the continuity of the Orlicz function $\Phi$, we get
\[
\int_{\mathbb{H}^n} \Phi \left( \| f \|^{-1}_{\Phi} |f(z)|  \right) dz \leq 
\limsup_{j \to \infty} \int_{\mathbb{H}^n} \Phi \left( \alpha^{-1}_j |f(z)|  \right) dz \leq 1.
\]
This finishes the proof.
\end{proof}

\begin{definition} \label{positive type}
An Orlicz function $\Phi : [0, \infty) \to [0, \infty)$ is said to be of positive lower (respectively, positive upper) type $p$ with 
$p \in (0, \infty)$ if there exists a positive constant $C$, depending on $p$, such that, for any $t >0$ and $r \in (0,1]$ (respectively, 
$r \in [1, \infty)$),

\begin{equation} \label{lower upper type}
\Phi(rt) \leq C r^p \Phi(t).
\end{equation}
\end{definition}

\begin{lemma} \label{lower estim}
Let $\Phi$ be an Orlicz function with positive lower type $p_{\Phi}^{-}$ and positive upper type $p_{\Phi}^{+}$. Then,

(i) $\Phi(t) \gtrsim \, t^{p_{\Phi}^{-}}$, if \, $t \geq 1$,

(ii) $\Phi(t) \gtrsim \, t^{p_{\Phi}^{+}}$, if \, $0 < t \leq 1$.
\end{lemma}

\begin{proof}
Write $\Phi(1) = \Phi(t^{-1} t)$ and apply (\ref{lower upper type}) according to the case.
\end{proof}

Given an Orlicz function $\Phi$, for any measurable function $f$ on $\mathbb{H}^n$, we set
\[
\kappa_{\Phi}(f) := \int_{\mathbb{H}^n} \Phi(|f(z)|) dz.
\]

\begin{lemma} \label{modular 1}
Let $\Phi$ be an Orlicz function of positive upper type $p_{\Phi}^{+}$. Then, $f \in L^{\Phi}(\mathbb{H}^n)$ if and only if 
$\kappa_{\Phi}(f) < \infty$.
\end{lemma}

\begin{proof}
Assume that $f \neq 0$. If $\kappa_{\Phi}(f) < \infty$, by the monotone convergence theorem, we obtain $f \in L^{\Phi}(\mathbb{H}^n)$. Now, if 
$f \in L^{\Phi}(\mathbb{H}^n)$, there exists $\lambda > 1$ such that $\kappa_{\Phi}(f/\lambda) \leq 1$. Then, (\ref{lower upper type}) leads to
\[
\kappa_{\Phi}(f) = \kappa_{\Phi}(\lambda f/ \lambda) \leq C \lambda^{p_{\Phi}^{+}} \kappa_{\Phi}(f/\lambda) \leq C \lambda^{p_{\Phi}^{+}} < \infty. 
\]
This concludes the proof.
\end{proof}

\begin{lemma} \label{modular 2}
Let $\Phi$ be an Orlicz function of positive upper type $p_{\Phi}^{+}$. If $\{ f_j \}$ is a sequence of measurable functions on 
$\mathbb{H}^n$ such that $\kappa_{\Phi}(f_j) \to 0$, then $\| f_j \|_{\Phi} \to 0$.
\end{lemma}

\begin{proof}
Suppose that $0< \kappa_{\Phi}(f_j) \to 0$. Given $0 < \epsilon < 1$ and $C > 0$ as in (\ref{lower upper type}), for all sufficiently 
large $j$, we have $0 < C \kappa_{\Phi}(f_j) < \epsilon^{p_{\Phi}^{+}}$ and
\[
\kappa_{\Phi} \left( C^{-1/p_{\Phi}^{+}} \kappa_{\Phi}(f_j)^{-1/p_{\Phi}^{+}} f_j \right)  \leq \kappa_{\Phi}(f_j)^{-1} \kappa_{\Phi}(f_j) = 1, 
\]
so $\| f_j \|_{\Phi} \leq C^{1/p_{\Phi}^{+}} \kappa_{\Phi}(f_j)^{1/p_{\Phi}^{+}} < \epsilon$. Then, $\| f_j \|_{\Phi} \to 0$.
\end{proof}

\begin{lemma} \label{lower upper fits}
Assume that $s \in (0, \infty)$. If $\Phi$ is an Orlicz function of positive lower (resp., positive upper) type $p_{\Phi}^{-}$ 
(resp., type $p_{\Phi}^{+}$),  then $\Phi_s$ is of positive lower (resp., positive upper) type  $s p_{\Phi}^{-}$ 
(resp., type $s p_{\Phi}^{+}$). 
\end{lemma}

\begin{proof}
It follows immediately from the definitions.
\end{proof}

\begin{remark}
Writing $\Phi(t) = \Phi(r r^{-1} t)$ with $r \in (0,1)$, it is easy to check that if an Orlicz function 
$\Phi$ is both of positive lower type $p_{\Phi}^{-}$ and of positive upper type $p_{\Phi}^{+}$, then $p_{\Phi}^{-} \leq p_{\Phi}^{+}$. 
Moreover, if $\Phi$ is of positive lower (resp., positive upper) type $p_{\Phi}^{-}$ 
(resp., type $p_{\Phi}^{+}$), then it also is of positive lower (resp., positive upper) type $p$ for any $p \in (0, p_{\Phi}^{-})$ 
(resp., type $p$ for any $p \in (p_{\Phi}^{+}, \infty)$).
\end{remark}

This remark leads to the following definition. Given an Orlicz function $\Phi$, as in \cite[Section 2.1]{Ky}, define
\begin{equation} \label{p menos}
i(\Phi) := \sup \{ p_{\Phi}^{-} : \Phi \,\, \text{is of positive lower type} \,\, p_{\Phi}^{-} \},
\end{equation}
and
\begin{equation} \label{p mas}
I(\Phi) := \inf \{ p_{\Phi}^{+} : \Phi \,\, \text{is of positive upper type} \,\, p_{\Phi}^{+} \}.
\end{equation}

The following result states that the couple $\left(L^{\Phi}(\mathbb{H}^n), \Vert \cdot \Vert_{\Phi} \right)$ is a quasi-normed space 
when $\Phi$ is an Orlicz function of positive lower type $p_{\Phi}^{-}$ and of positive upper type $p_{\Phi}^{+}$. In particular, if 
$\Phi(t) = t^p$, $0 < p < \infty$, then $L^{\Phi}(\mathbb{H}^n) = L^p(\mathbb{H}^n)$ with quasi-norm coincidence.

\begin{proposition} \label{quasi-norm}
Let $\Phi$ be an Orlicz function with positive lower type $p_{\Phi}^{-}$ and positive upper type $p_{\Phi}^{+}$, then for any 
$f, g : \mathbb{H}^n \to \mathbb{C}$ belonging to $L^{\Phi}(\mathbb{H}^n)$ and $\alpha \in \mathbb{C}$,

(i) $\Vert f \Vert_{\Phi} \geq 0$ and $\Vert f \Vert_{\Phi} = 0$ if and only if $f(z) = 0$ a.e. $z$;

(ii) $\Vert \alpha f \Vert_{\Phi} = \vert \alpha \vert \Vert f \Vert_{\Phi}$;

(iii) $\Vert f + g \Vert_{\Phi} \leq K (\Vert f \Vert_{\Phi} + \Vert g \Vert_{\Phi})$, where $K \geq 1$ and does not depend on $f$ and $g$;

(iv) If $0 \leq f \leq g$ a.e., then $\| f \|_{\Phi} \leq \| g \|_{\Phi}$;

(v) If $\Phi$ is continuous and $0 \leq f_n \uparrow f$ a.e., then $\| f_n \|_{{\Phi}} \uparrow \| f \|_{{\Phi}}$;

(vi) If $\Phi$ is bijective and $E \subset \mathbb{H}^n$ is measurable with $|E| < \infty$, then 
$\| \chi_E  \|_{\Phi} = \frac{1}{\Phi^{-1}(|E|^{-1})}$.
\end{proposition}

\begin{proof} The first part of (i) is obvious, the second one it follows from that $\Phi$ is an Orlicz function of positive lower 
type $p_{\Phi}^{-}$. From the definition of $\Vert \cdot \Vert_{\Phi}$ follows (ii) and (vi). Now, (iii) is consequence of 
\cite[Lemma 2.3]{Zhang}. As $\Phi$ is non-decreasing, (iv) follows. Finally, (v) follows from monotone convergence Theorem.
\end{proof}

\begin{definition}
Two Orlicz functions $\Phi$ and $\Psi$ are called equivalent, denoted by $\Phi \sim \Psi$, if there exist constants $C_0 \geq 1$ and 
$C \geq 1$ such that
\[
C_{0}^{-1}\Phi(t/C) \leq \Psi(t) \leq C_0 \Phi(Ct),
\]
for all $t \geq 0$.
\end{definition}

\begin{remark} \label{cont and crec}
It is clear that if $\Phi \sim \Psi$, then $L^{\Phi} = L^{\Psi}$ with equivalent Luxemburg norms. We observe that all our results are 
invariant under the change of equivalent Orlicz functions. Moreover, equivalent Orlicz functions have the same positive lower and upper type numbers. By \cite[Lemma 2.5]{Zhang}, without loss of generality, we may always assume that an Orlicz function $\Phi$ of positive lower type $p_{\Phi}^{-}$ and positive upper type $p_{\Phi}^{+}$ is continuous and strictly increasing.
\end{remark}

\begin{corollary}
Let $\Phi$ be an Orlicz function with positive lower type $p_{\Phi}^{-}$ and positive upper type $p_{\Phi}^{+}$, then there exists an Orlicz function $\Psi$ equivalent to $\Phi$ such that $L^{\Psi}(\mathbb{H}^n)$ is a quasi-Banach function space.
\end{corollary}

\begin{proof}
By \cite[Definitions 2.8 and 2.9]{Nekvinda}, the corollary follows from Proposition \ref{quasi-norm} and Remark \ref{cont and crec}.
\end{proof}

\begin{definition}
A function $\Phi : [0, \infty) \to [0, \infty)$ is called a \textit{Young function} if $\Phi$ is convex, left-continuous, $\lim_{t \to 0^{+}} \Phi(t) = \Phi(0) = 0$, and $\lim_{t \to \infty} \Phi(t) = \infty$.
\end{definition}

\begin{remark}
From the convexity and $\Phi(0) = 0$, it follows that any Young function is non-decreasing. Moreover, if $\Phi$ is a young function, from \cite[Lemma 3.2.2 (b) and Theorem 3.3.7 (b)]{Hasto}, it follows that the couple 
$(L^{\Phi}, \| \cdot \|_{\Phi})$ is a Banach space.
\end{remark}

\begin{remark} \label{Orlicz cont}
If an Orlicz function $\Phi$ is also a Young function, then $\Phi$ is a bijective continuous function from $[0, \infty)$ onto $[0, \infty)$.
\end{remark}

For a Young function $\Phi$, we define $\Phi^{-1}$ and its complementary function $\Phi^{*}$ on $[0, \infty)$ by
\[
\Phi^{-1}(s):= \inf\{ t \geq 0 : \Phi(t) > s \}
\]
and
\[
\Phi^{*}(s) := \sup\{ ts - \Phi(t) : t \in [0, \infty) \},
\]
respectively. From the definition of $\Phi^{*}$, it follows that
\begin{equation} \label{Young ineq}
ts \leq \Phi(t) + \Phi^{*}(s), \,\,\,\, \forall t,s \geq 0,
\end{equation}
Then, the K\"othe dual $(L^{\Phi})' = L^{\Phi^{*}}$ with comparable norms (see \cite[Lemma 2.4.2 and Theorem 3.4.6]{Hasto}). Moreover, 
by \cite[Property 1.6]{ONeil}, we have that
\begin{equation} \label{prop 1.6}
s \leq \Phi^{-1}(s) (\Phi^{*})^{-1}(s) \leq 2s, \,\,\,\, s \geq 0.
\end{equation}
If $\Phi$ is a Young-Orlicz function, then $\Phi^{-1}$ is the usual inverse function of $\Phi$.

\begin{definition}
An Orlicz function $\Phi$ is called an $N$-function if it is a continuous and convex function such that
\begin{equation} \label{2 limites}
\lim_{t \to 0^{+}} \frac{\Phi(t)}{t} = 0, \,\,\,\,\,\, \text{and} \,\,\,\,\,\,  \lim_{t \to \infty} \frac{\Phi(t)}{t} = \infty.
\end{equation}
\end{definition}

\begin{remark} \label{N funct}
From \cite[Lemma 2.16]{Zhang} we have that if $\Phi$ is an Orlicz function of positive lower type $p_{\Phi}^{-} \in (1, \infty)$ and 
positive upper type $p_{\Phi}^{+}$, then there exists an Orlicz $N$-function $\Psi$ equivalent to $\Phi$ of positive lower type 
$p_{\Psi}^{-} = p_{\Phi}^{-}$ and positive upper type $p_{\Psi}^{+} = p_{\Phi}^{+}$. Thus, without loss of generality, we may always 
assume that an Orlicz function $\Phi$ of positive lower type $p_{\Phi}^{-}$ and positive upper type $p_{\Phi}^{+}$, with 
$1 < p_{\Phi}^{-} \leq p_{\Phi}^{+} < \infty$, is also an $N$-function. In particular, an Orlicz $N$-function is a Young-Orlicz function.
\end{remark}

Next, we will prove that the Hardy-Littlewood maximal operator $M$ on the Heisenberg group is bounded on Orlicz spaces 
$L^{\Phi}(\mathbb{H}^n)$, when $\Phi$ is an Orlicz function with positive lower type $p_{\Phi}^{-} > 1$ and positive upper type 
$p_{\Phi}^{+}$. As an application of this result, we will obtain a vector-valued inequality for $M$. 

\begin{lemma} \label{N compl}
If $\Phi$ is an Orlicz $N$-function of positive upper type $p_{\Phi}^{+} > 1$, then its complementary function $\Phi^{*}$ is also an Orlicz 
$N$-function.
\end{lemma}

\begin{proof}
It is clear that $\Phi$ is a Young-Orlicz function. By \cite[Lemma 2.4.2]{Hasto}, $\Phi^{*}$ is a Young function. From (\ref{2 limites}), it follows that $\Phi^{*}(s) > 0$ for all $s >0$ and so $\Phi^{*}$ is continuous (see Remark \ref{Orlicz cont}). Now, from (\ref{Young ineq}), 
we have that $\Phi^{*}(s)/s \to \infty$ as $s \to \infty$. As $\Phi$ is of positive upper type $p_{\Phi}^{+} > 1$, by 
\cite[Proposition 7.8]{sawa}, we have that $\Phi^{*}$ is of positive lower type $(p_{\Phi}^{+})' > 1$ and thus, for $0 < s < 1$, results
\[
0 \leq \frac{\Phi^{*}(s)}{s} \leq C s^{(p_{\Phi}^{+})' - 1} \Phi^{*}(1) \to 0, \,\,\,\, \text{as} \,\, s \to 0.
\]
Therefore, $\Phi^{*}$ is an Orlicz $N$-function.
\end{proof}

\begin{proposition} \label{M bound 1}
If $\Phi$ is an Orlicz function of positive lower type $p_{\Phi}^{-}$ and upper type $p_{\Phi}^{+}$ with 
$1 < p_{\Phi}^{-} \leq p_{\Phi}^{+} < \infty$, then the Hardy-Littlewood maximal operator $M$ is bounded on $L^{\Phi}(\mathbb{H}^n)$.
\end{proposition}

\begin{proof}
By Remark \ref{N funct}, we can assume that $\Phi$ is an Orlicz $N$-function with $1 < p_{\Phi}^{-} \leq p_{\Phi}^{+} < \infty$. From Lemma 
\ref{N compl}, it follows that $\Phi^{*}$ is an Orlicz $N$-function. Since the Hardy-Littlewood maximal operator on $\mathbb{H}^n$ is 
of weak type $(1, 1)$ and type $(\infty, \infty)$ (see \cite[{\bf 2.5} and Theorem 1 in Chapter I]{Elias}), the proposition then follows 
from \cite[Theorem 2.2]{Gallardo}.
\end{proof}

\begin{corollary} \label{M bound 2}
If $\Phi$ is an Orlicz function of positive lower type $p_{\Phi}^{-}$ and upper type $p_{\Phi}^{+}$ with 
$1 < p_{\Phi}^{-} \leq p_{\Phi}^{+} < \infty$, then the Hardy-Littlewood maximal operator $M$ is bounded on $L^{\Phi^{*}}(\mathbb{H}^n)$, where $\Phi^{*}$ is the complementary function of $\Phi$.
\end{corollary}

\begin{proof}
Assuming that $\Phi$ is a $N$-function, it follows, by Lemma \ref{N compl} and \cite[Proposition 7.8]{sawa}, that $\Phi^{*}$ is 
an Orlicz $N$-function of positive lower type $p_{\Phi^{*}}^{-} = (p_{\Phi}^{+})' > 1$ and upper type 
$p_{\Phi^{*}}^{+} = (p_{\Phi}^{-})' \geq (p_{\Phi}^{+})'$. Then, the corollary follows from Proposition \ref{M bound 1} applied 
to $\Phi^{*}$.  
\end{proof}

\begin{theorem} \label{Feff-Stein max ineq}
Let $\Phi$ be an Orlicz function of positive lower type $p_{\Phi}^{-}$ and upper type $p_{\Phi}^{+}$ with 
$1 < p_{\Phi}^{-} \leq p_{\Phi}^{+} < \infty$. Then, for any $1 < r < \infty$,
\begin{equation} \label{vector ineq M}
\left\| \left( \sum_{j=1}^{\infty} (M f_j)^r  \right)^{1/r}  \right\|_{\Phi} \lesssim 
\left\| \left( \sum_{j=1}^{\infty} |f_j|^r  \right)^{1/r}  \right\|_{\Phi}
\end{equation}
holds for all sequences of locally integrable functions $\{ f_j \}_{j=1}^{\infty}$ on $\mathbb{H}^n$.
\end{theorem}
 
\begin{proof}
Given $1 < p < \infty$, by \cite[Proposition 7.13]{Hytonen}, we have, for any weight $w \in \mathcal{A}_p(\mathbb{H}^n)$ 
(see \cite[p. 28]{Hytonen}), that the Hardy-Littlewood maximal operator $M$ is a bounded operator 
$L^p_w (\mathbb{H}^n) \to L^p_w (\mathbb{H}^n)$. Then, combining Proposition \ref{M bound 1}, Corollary \ref{M bound 2} and proceeding as in the proof of \cite[Theorem 2.19]{sawa} (we point out that such argument works as well on $\mathbb{H}^n$), (\ref{vector ineq M}) follows.
\end{proof}

We finish this section with the following auxiliary result. 

\begin{proposition} \label{b_j functions}
Let $\Phi$ be an Orlicz function of positive lower type $p_{\Phi}^{-}$ and positive upper type $p_{\Phi}^{+}$. Suppose that $s > 1$ and 
$0 < \theta < \min\{ 1, p_{\Phi}^{-} \}$ are such that $s \theta > p_{\Phi}^{+}$ and $\{ b_j \}_{j=1}^{\infty}$ is a sequence of non-negative functions in $L^{s}(\mathbb{H}^{n})$ such that each $b_j$ is supported in a $\rho$ - ball $B_j \subset \mathbb{H}^{n}$ and
\[
\| b_j \|_{L^{s}(\mathbb{H}^{n})} \leq A_j |B_j|^{1/s},
\]
where $A_j >0$ for all $j \geq 1$. Then, for any sequence of non-negative numbers $\{ k_j \}_{j=1}^{\infty}$ we have
\[
\left\| \sum_{j=1}^{\infty} k_j b_j \right\|_{\Phi_{1/\theta}} \leq C \left\| \sum_{j=1}^{\infty} A_j k_j \chi_{B_j} 
\right\|_{\Phi_{1/\theta}},
\]
where $C$ is a positive constant which does not depend on $\{ b_j \}_{j=1}^{\infty}$, $\{ A_j \}_{j=1}^{\infty}$, and 
$\{ k_j \}_{j=1}^{\infty}$.
\end{proposition}

\begin{proof}
The argument used to proof \cite[Proposition 3.3]{Pablo} also works in this setting, but now considering the space $L^{\Phi}(\mathbb{H}^n)$, 
\cite[Lemma 3.2.11]{Hasto}, Lemma \ref{fits} and Corollary \ref{M bound 2} applied with $\left( (\Phi^{1/\theta})^{*} \right)^{1/s'}$ and taking into account that $1 < s' < (p_{\Phi}^{+}/\theta)' \leq (p_{\Phi}^{-}/\theta)'$. 
\end{proof}

\section{Orlicz-Hardy spaces on the Heisenberg group} \label{OH}

In this section, we define the Orlicz-Hardy spaces $H^{\Phi}(\mathbb{H}^n)$ where $\Phi$ is an Orlicz function of positive lower type 
$p_{\Phi}^{-}$ and positive upper type $p_{\Phi}^{+}$. We also provide an atomic decomposition for elements in $H^{\Phi}(\mathbb{H}^n)$.

Given $N \in \mathbb{N}$, define 
\[
\mathcal{F}_{N}=\left\{ \varphi \in \mathcal{S}(\mathbb{H}^{n}) : \| \phi \|_{\mathcal{S}(\mathbb{H}^{n}), \, N} \leq 1\right\}.
\]
For any $f \in \mathcal{S}'(\mathbb{H}^{n})$, the grand maximal function of $f$ is defined by 
\[
\mathcal{M}_N f(z)=\sup\limits_{t>0}\sup\limits_{\phi \in \mathcal{F}_{N}}\left\vert \left( f \ast \phi_t \right)(z) \right\vert,
\]
where $\phi_t(z) := t^{-2n-2} \phi(t^{-1} \cdot z)$ with $t > 0$. 

Now, we introduce two maximal functions. Given $\phi \in \mathcal{S}(\mathbb{H}^n)$ and $f \in \mathcal{S}'(\mathbb{H}^n)$, we define the discrete maximal function $M_{\phi}^{dis}f$ by
\[
(M_{\phi}^{dis}f)(z) := \sup \left\{ |(f \ast \phi_{2^{-j}})(z)| : j \in \mathbb{Z} \right\}, \,\,\,\, z \in \mathbb{H}^n.
\]
Given an integer $L > 1$, we define the maximal function $M_{\phi, L}^{*}f$ by
\[
M_{\phi, L}^{*}f(z) := \sup_{j \in \mathbb{Z}} \sup_{w \in \mathbb{H}^n} \frac{|(f \ast \phi_{2^{-j}})(w)|}{(1+ 4^{j} \rho(z^{-1} \cdot w)^2)^L}, 
\,\,\,\, z \in \mathbb{H}^n.
\]

\begin{lemma} (\cite[Lemma 3.4]{Fang}) \label{M star y phi}
Let $f \in \mathcal{S}'(\mathbb{H}^n)$, $0 < \theta < 1$ and let $\phi$ be a radial function in $\mathcal{S}(\mathbb{H}^n)$ with 
$\int \phi \neq 0$. Then, there exists $L_{\theta}$ such that for all $L \geq L_{\theta}$ and $z \in \mathbb{H}^n$,
\[
M_{\phi, L}^{*}f(z) \lesssim \left[ M \left( (M_{\phi}^{dis}f)^{\theta} \right)(z) \right]^{1/\theta}.
\]
\end{lemma}

In the next result, we show that, for $N$ and $L$ large enough, the quantities $\| \mathcal{M}_{N}f \|_{\Phi}$, 
$\| M_{\phi}^{dis}f \|_{\Phi}$ and $\| M_{\phi, L}^{*} f \|_{\Phi}$ are mutually comparable, with bounds independent of $f$.

\begin{theorem} \label{norm comparable}
Let $\Phi$ be an Orlicz function such that $0 < i(\Phi) \leq I(\Phi) < \infty$. For a radial function $\phi \in \mathcal{S}(\mathbb{H}^n)$ with $\int \phi \neq 0$, we have
\[
\| \mathcal{M}_{N}f \|_{\Phi} \approx \| M_{\phi, L}^{*} f \|_{\Phi} \approx \| M_{\phi}^{dis}f \|_{\Phi},
\]
for all $f \in \mathcal{S}'(\mathbb{H}^n)$, where $N$ and $L$ are large enough.
\end{theorem}

\begin{proof}
It is clear that $M_{\phi}^{dis}f(z) \leq M_{\phi, L}^{*} f(z)$ for all $z \in \mathbb{H}^n$. Now, for $0 < \theta < \min \{1, i(\Phi) \}$, from Lemma \ref{M star y phi}, Lemma \ref{fits} and Proposition \ref{M bound 1}, it follows that
\[
\| M_{\phi, L}^{*} f \|_{\Phi} \approx \| M_{\phi}^{dis}f \|_{\Phi}.
\]
On the other hand, we have that $M_{\phi}^{dis}f(z) \leq \mathcal{M}_N f(z)$ for all $z \in \mathbb{H}^n$. Thus,
\[ 
\| M_{\phi, L}^{*} f \|_{\Phi} \lesssim \| \mathcal{M}_N f \|_{\Phi}.
\]
In the proof of \cite[Theorem 3.2]{Fang}, the authors shown, for $N$ and $L$ large enough 
and $\tau \in \mathcal{S}(\mathbb{H}^n)$ satisfying $\| \tau \|_{\mathcal{S}(\mathbb{H}^{n}), \, N} \leq 1$, that
\[
|(f \ast \tau_{2^{-j}})(z)| \lesssim M_{\phi, L}^{*} f (z), \,\,\,\, \text{for all} \,\, z \in \mathbb{H}^n.
\]
Being $\tau$ and $j$ arbitrary, one obtains $\mathcal{M}_N f(z) \lesssim M_{\phi, L}^{*} f (z)$ for all $z \in \mathbb{H}^n$, and with them 
the theorem.
\end{proof}

\begin{definition} \label{Orlicz Hardy space}
Given an Orlicz funtion $\Phi$ of positive lower type $p_{\Phi}^{-}$ and positive upper type $p_{\Phi}^{+}$, we define
the Orlicz-Hardy space $H^{\Phi}(\mathbb{H}^{n})$ as the set of all $f \in S^{\prime}(\mathbb{H}^{n})$ for which 
$\mathcal{M}_{N}f \in L^{\Phi}(\mathbb{H}^{n})$, where $N$ is large enough in the sense of Theorem \ref{norm comparable}. In this case 
we set $\| f \|_{H^{\Phi}(\mathbb{H}^{n})} = \| \mathcal{M}_{N}f \|_{\Phi}$.
\end{definition}

Now, we introduce the definition of $\Phi$-atom in $\mathbb{H}^n$.

\begin{definition} \label{atom def}
Let $\Phi$ be an Orlicz function of positive lower type $p_{\Phi}^{-}$ and positive upper type $p_{\Phi}^{+}$ with 
$0 < i(\Phi) \leq I(\Phi) < \infty$, $\max\{ 1, I(\Phi)  \} < p_0 \leq \infty$ and $m \in \mathbb{N}$. Fix an integer 
$m \geq m_{\Phi} : = Q \left(\lfloor i(\Phi)^{-1} - 1 \rfloor + 1 \right)$. A measurable function $a(\cdot)$ 
on $\mathbb{H}^{n}$ is called a $(\Phi, p_{0}, m)$ - atom if there exists a $\rho$ - ball $B$ such that \newline
$a_{1})$ $\textit{supp}\left( a\right) \subset B$, \newline
$a_{2})$ $\left\Vert a \right\Vert_{L^{p_{0}}(\mathbb{H}^{n})} \leq 
\left\vert B \right\vert^{\frac{1}{p_{0}}} \Vert \chi_B \Vert_{\Phi}^{-1}$, \newline
$a_{3})$ $\int a(z) \, z^{I} \, dz = 0$ for all multiindex $I$ such that $d(I) \leq m$.
\end{definition}
A such atom is also called an atom centered at the $\rho$ - ball $B$. Following the ideas in the proof of \cite[Proposition 4.1]{Fang}, and adapting them to our context, one obtains that every $(\Phi, p_{0}, m)$ - atom 
$a(\cdot)$ belongs to $H^{\Phi}(\mathbb{H}^{n})$. Moreover, there exists an universal constant $C > 0$ 
such that $\| a \|_{H^{\Phi}(\mathbb{H}^n)} \leq C$ for all $(\Phi, p_{0}, m)$ - atom $a(\cdot)$.

\begin{remark} \label{atomo infinito}
We observe that every $(\Phi, \infty, m)$ - atom is an $(\Phi, p_0, m)$ - atom for any $p_0 \in (1, \infty)$.
\end{remark}

Now, we shall formulate an atomic decomposition theorem in terms of $(\Phi, \infty, m)$ - atoms. Before establishing this result, we introduce the following three constants (see \cite[p. 81]{Foll-St})
\[
T_1 = 9 \gamma \beta^{N}, \,\,\,\,\,\, T_2 = 2 \gamma^2 T_1, \,\,\,\, \text{and} \,\,\,\, T_3 = 3 \gamma T_2,
\] 
where $\gamma$ is the constant in the inequality \cite[(1.8)]{Foll-St} (when $G = \mathbb{H}^n$ and $|\cdot|=\rho(\cdot)$ is the Koranyi norm given by (\ref{Koranyi norm}), $\gamma = 1$), $\beta$ is  the constant in \cite[Theorem 1.33]{Foll-St} (we observe that $\beta \geq 1$, see \cite[p. 29]{Foll-St}), and $N$ is as in Definition \ref{Orlicz Hardy space}.

\begin{theorem} \label{infinite atomic decomp}
Let $\Phi$ be an Orlicz function such that $0 < i(\Phi) \leq I(\Phi) < \infty$. Then every $f \in H^{\Phi}(\mathbb{H}^n)$ can be written as
\begin{equation}
f=\sum\limits_{j=1}^{\infty } \lambda_{j}a_{j}  \label{serie atomica}
\end{equation}
in $S^{\prime }(\mathbb{H}^{n}),$ where $\left\{ \lambda_{j} \right\}_{j=1}^{\infty}$ is a sequence of non-negative numbers, the $a_{j}$'s are 
$(\Phi, \infty, m)$ - atoms supported on $\rho$ - balls $B_j$ and
\begin{equation} \label{norma atomica}
\left\Vert \left\{ \sum_{j} \left( \frac{\lambda_j \chi_{B_j}}{\Vert \chi_{B_j} \Vert_{\Phi}} \right)^{\theta} \right\}^{1/\theta} 
\right\Vert_{\Phi} \leq C \Vert f \vert_{H^{\Phi}}, \,\,\,\,\,\,\,\, \forall \,\,\, 0 < \theta \leq 1,
\end{equation}
where $C$ is an universal positive constant which does not depend on $\left\{ \lambda_{j}\right\}_{j=1}^{\infty}$, 
$\left\{ B_{j}\right\}_{j=1}^{\infty}$ and $f$.
\end{theorem}

\begin{proof}
Let $f \in H^{\Phi}(\mathbb{H}^n) \cap L^{s}(\mathbb{H}^n)$ with $s \in (\max \{I(\Phi), 1\},\infty)$ and $m \geq m_{\Phi}$. For any $j \in \mathbb{Z}$ and $N$ large enough, we set
\[
O_j = \{ z \in \mathbb{H}^n : \mathcal{M}_N f(z) > 2^{j} \}.
\]
It is clear that $O_j \supseteq O_{j+1}$ for all $j \in \mathbb{Z}$. For every $j \in \mathbb{Z}$, by \cite[Whitney Lemma 1.67]{Foll-St} and taking into account the constants $T_1$, $T_2$ and $T_3$ previously defined, there exist a sequence of points $\{ z_{j,k} \}_{k \in \mathbb{N}} \subset O_j$ and a sequence of scalars $\{ r_{j,k} \}_{k \in \mathbb{N}} \subset (0, \infty)$ such that

(w1) $O_j = \cup_{k=1}^{\infty} B(z_{j,k}, r_{j,k})$,

(w2) the family of balls $\{ B(z_{j,k}, r_{j,k}/4) \}_{k \in \mathbb{N}}$ are disjoint,

(w3) $B(z_{j,k}, T_2 r_{j,k}) \cap O_j^c = \emptyset$, but $B(z_{j,k}, T_3 r_{j,k}) \cap O_j^c \neq \emptyset$,

(w4) there exists $L \in \mathbb{N}$ such that no point of $O_j$ lies in more than $L$ of the balls $B(z_{j,k}, T_2 r_{j,k})$.

Now, by Calder\'on-Zygmund decomposition given in \cite[Chapter 3]{Foll-St}, we can decomposed to $f$ as
\[
f = b_j  + g_j, \,\,\,\,\,\,\,\, b_j = \sum_k b_{j,k}, \,\,\,\,\,\,\,\, b_{j,k} = (f - P_{j,k}) \zeta_{j,k},
\]
where $\zeta_{j,k} \in C^{\infty}_{0}(B(z_{j,k}, 2 r_{j,k}))$, $\int b_{j,k}(z) z^{I} dz = 0$ for all $d(I) \leq m$ and $|g_j| \lesssim 2^j$. We put $\widetilde{B}_{j,k} := B(z_{j,k}, T_1 r_{j,k})$. By \cite[Lemmas 3.12 and 3.13]{Foll-St}, we have that
\[
\mathcal{M}_N b_{j,k}(z) \lesssim \mathcal{M}_N f(z), \,\,\,\,\, \text{if} \,\, z \in \widetilde{B}_{j,k},
\]
and 
\[
\mathcal{M}_N b_{j,k}(z) \lesssim 2^{j} \left( r_{j,k}/\rho(z_{j,k}^{-1} \cdot z) \right)^{Q+m}, \,\,\,\,\, \text{if} \,\, z \notin 
\widetilde{B}_{j,k}.
\]
So,
\[
\| \mathcal{M}_N b_j \|_{\Phi} \lesssim \left\| \sum_k \mathcal{M}_N b_{j, k} \right\|_{\Phi} \lesssim 
\left\| \sum_k \mathcal{M}_N f \cdot \chi_{\widetilde{B}_{j,k}} \right\|_{\Phi}
\]
\[ 
+ \left\| \sum_k 2^{j} \left( r_{j,k}/\rho(z_{j,k}^{-1} \cdot (\cdot)) \right)^{Q+m} \cdot \chi_{\widetilde{B}_{j,k}^{c}} \right\|_{\Phi}
\]
\[
\lesssim \left\| \mathcal{M}_N f \cdot \chi_{O_j} \right\|_{\Phi} +
\left\| \sum_k 2^{j} \left( M\chi_{\widetilde{B}_{j,k}} \right)^{\frac{Q+m}{Q}} \right\|_{\Phi}
\]
Since $m \geq m_{\Phi}$, from Theorem \ref{Feff-Stein max ineq}, it follows that
\[
\| \mathcal{M}_N b_j \|_{\Phi} \lesssim \left\| \mathcal{M}_N f \cdot \chi_{O_j} \right\|_{\Phi},
\]
and so, by Lemma \ref{modular 2},
\[
\| f - g_j \|_{H^{\Phi}(\mathbb{H}^n)} = \| b_j \|_{H^{\Phi}(\mathbb{H}^n)} = \| \mathcal{M}_N b_j \|_{\Phi} \lesssim  
\left\| \mathcal{M}_N f \cdot \chi_{O_j} \right\|_{\Phi} \to 0
\]
as $j \to +\infty$. On the other hand, we have $g_j \to 0$ uniformly as $j \to -\infty$. Then,
\[
f = \sum_{j=-\infty}^{+\infty} (g_{j+1} - g_j) \,\,\,\,\, \text{in} \,\, \mathcal{S}'(\mathbb{H}^n).
\]
Now, by \cite[p. 101-102]{Foll-St}, we can write $g_{j+1} - g_j = \sum_{k=1}^{\infty} h_{j,k}$ and $f = \sum_{j,k} h_{j,k}$ in 
$\mathcal{S}'(\mathbb{H}^n)$. Moreover, $h_{j,k}$ is supported on the ball $B(z_{j,k}, T_2 r_{j,k}) =: T_2 B_{j,k}$, $| h_{j,k} (z)| \leq C 2^{j}$ for 
all $z$, and $\int h_{j,k}(z) z^{I} dz = 0$ for all $d(I) \leq m$. Putting
\begin{equation} \label{ajk and lambdajk}
a_{j,k} := \frac{h_{j,k}}{\lambda_{j,k}}, \,\,\,\, \text{where} \,\, \lambda_{j,k} := C 2^{j} \| \chi_{T_2 B_{j,k}} \|_{\Phi},
\end{equation}
we have that the $a_{j,k}$'s are $(\Phi, \infty, m)$ - atoms and $f = \sum_{j,k} \lambda_{j,k} a_{j,k}$ in $\mathcal{S}'(\mathbb{H}^n)$. By 
(w1), (w4) and (\ref{ajk and lambdajk}), it results
\[
\left\| \left\{ \sum_{j, k} \left( \frac{\lambda_{j,k} \chi_{T_2 B_{j,k}}}{\| \chi_{T_2 B_{j,k}} \|_{\Phi}} \right)^{\theta} 
\right\}^{1/\theta} \right\|_{\Phi} \approx \left\| \left\{ \sum_{j=-\infty}^{+\infty} 2^{j \theta} \chi_{O_j} \right\}^{1/\theta} 
\right\|_{\Phi}
\]
Since $O_j \supseteq O_{j+1}$ for all $j \in \mathbb{Z}$ and $0 < \theta \leq 1$, we have that
\[
\sum_{j=-\infty}^{+\infty} \left( 2^j \chi_{O_j}(z)  \right)^{\theta} \approx
\left( \sum_{j=-\infty}^{+\infty} 2^j \chi_{O_j}(z) \right)^{\theta} \approx
\left( \sum_{j=-\infty}^{+\infty} 2^j \chi_{O_j \setminus O_{j+1}}(z) \right)^{\theta},
\]
for all $z \in \mathbb{H}^n$. So,
\[
\left\| \left\{ \sum_{j, k} \left( \frac{\lambda_{j,k} \chi_{T_2 B_{j,k}}}{\| \chi_{T_2 B_{j,k}} \|_{\Phi}} \right)^{\theta} 
\right\}^{1/\theta} \right\|_{\Phi} \lesssim
\left\| \sum_{j=-\infty}^{+\infty} 2^j \chi_{O_j \setminus O_{j+1}} \right\|_{\Phi}
\]
Now, taking into account that the family of sets $\{ O_j \setminus O_{j+1} \}_{j \in \mathbb{Z}}$ is disjoint, for any $\mu > 0$ we have that
\[
\int_{\mathbb{H}^n} \Phi\left( \sum_{j=-\infty}^{+\infty} \frac{2^j \chi_{O_j \setminus O_{j+1}}(z)}{\mu} \right) dz =
\sum_{j=-\infty}^{+\infty} \int_{O_j \setminus O_{j+1}} \Phi\left(  \frac{2^j}{\mu} \right) dz \approx
\int_{\mathbb{H}^n} \Phi\left( \frac{\mathcal{M}_N f(z)}{\mu} \right) dz,
\]
which leads to
\[
\left\| \left\{ \sum_{j, k} \left( \frac{\lambda_{j,k} \chi_{T_2 B_{j,k}}}{\| \chi_{T_2 B_{j,k}} \|_{\Phi}} \right)^{\theta} 
\right\}^{1/\theta} \right\|_{\Phi} \lesssim
\| f \|_{\Phi}.
\]
Proceeding as in the proof of \cite[Lemma 4.13]{Nakai}, for $s \in (I(\Phi),\infty)$, one obtains that 
$H^{\Phi}(\mathbb{H}^n) \cap L^{s}(\mathbb{H}^n)$ is dense in $H^{\Phi}(\mathbb{H}^n)$. From this and an argument similar to 
\cite[p. 109]{Elias}, we have an atomic decomposition for any $f \in H^{\Phi}(\mathbb{H}^n)$ satisfying (\ref{norma atomica}).  
\end{proof}

The following auxiliary result will be useful to get our main theorem of Section \ref{main thm}.

\begin{proposition} \label{max on atoms}
Let $\Phi$ be an Orlicz function such that $0 < i(\Phi) \leq I(\Phi) < \infty$, and let
$\max\{ 1, I(\Phi)  \} < p_0 < \infty$ and $0 < \theta < \min\{ 1, i(\Phi) \}$. Then, for any non-negative sequence 
$\{ k_j \}_{j=1}^{\infty}$, $r \geq 1$, and any sequence $\{ a_j \}_{j=1}^{\infty}$ of $(\Phi, p_0, m)$-atoms such that every 
atom $a_j$ is supported on the $\rho$ - ball  $B_j$, we have
\[
\left\| \sum_{j=1}^{\infty} k_j \chi_{r B_j} M a_j \right\|_{\Phi} \lesssim 
\left\| \left\{ \sum_{j=1}^{\infty} \left( \frac{ k_j \chi_{B_j}}{\| \chi_{B_j} \|_{\Phi}} \right)^{\theta} 
\right\}^{1/\theta} \right\|_{\Phi},
\]
where $M$ is the Hardy-Littlewood maximal operator and the implicit constant does not depend on $\{ k_j \}_{j=1}^{\infty}$, 
$\{ a_j \}_{j=1}^{\infty}$, and $\{ B_j \}_{j=1}^{\infty}$.
\end{proposition}

\begin{proof}
Let $0 < \theta < \min\{1, i(\Phi) \}$ be fixed, $r \geq 1$ and $p_{0} > \max \{1, I(\Phi) \}$. Being the Hardy-Littlewood 
maximal operator $M$ is bounded on $L^{p_0}(\mathbb{H}^n)$ (see \cite[Theorem 1, p. 13]{Elias}), we have
\begin{equation} \label{Mfi}
\|  (Ma_j)^{\theta} \|_{L^{p_{0}/\theta}(r B_j)} \lesssim \| a_j \|_{p_0}^{\theta} \lesssim 
\frac{|B_j |^{\theta/p_0}}{\| \chi_{B_j} \|_{\Phi}^{\theta}} 
\lesssim \frac{ |r B_j |^{\theta/p_0}}{\| \chi_{ B_j} \|_{\Phi}^{\theta}}.
\end{equation}
Now, since $0 < \theta < 1$, we apply the $\theta$-inequality given in \cite[1.1.4 (a), on p. 12]{Grafakos} and Lemma \ref{fits}, to obtain
\[
\left\| \sum_{j} k_j \chi_{r B_j} M a_j \right\|_{\Phi} \leq 
\left\| \sum_{j} \left(k_j \, \chi_{r B_{j}} \, Ma_j \right)^{\theta} \right\|^{1/\theta}_{\Phi_{1/\theta}} =:J_{\theta}.
\]
Then, taking into account (\ref{Mfi}), by Proposition \ref{b_j functions} with $b_j = \left( \chi_{r B_{j}} \, (Ma_j)^{\theta} \right)$, 
$A_j = \| \chi_{ B_{j}} \|_{\Phi}^{-\theta}$ and $s= p_0/\theta$, we have that $J_{\theta}$
\begin{equation} \label{Jteta}
\lesssim \left\| \sum_{j} \left( \frac{k_j}{\left\| \chi_{ B_{j}} \right\|_{\Phi}} \right)^{\theta} 
\chi_{r B_{j}} \right\|^{1/\theta}_{\Phi_{1/\theta}}.
\end{equation}
It is easy to check that $\chi_{r B_{j}} \leq (M\chi_{B_j})^{2}$. From this pointwise inequality, the inequality (\ref{Jteta}), 
Theorem \ref{Feff-Stein max ineq} and Lemma \ref{fits}, we have
\begin{eqnarray*}
\left\| \sum_{j} k_j \chi_{r B_j} M a_j \right\|_{\Phi} &\lesssim& 
\left\| \left\{ \sum_{j} \left( \frac{k_j^{\theta/2}}{\left\| \chi_{B_{j}} \right\|^{\theta/2}_{\Phi}} 
(M\chi_{B_j}) \right)^{2}  \right\}^{1/2} \right\|^{2/\theta}_{\Phi_{2/\theta}} \\
&\lesssim& \left\| \left\{ \sum_{j} \left( \frac{ k_j \chi_{B_j}}{\| \chi_{B_j} \|_{\Phi}} \right)^{\theta} 
\right\}^{1/\theta} \right\|_{\Phi}.
\end{eqnarray*}
This concludes the proof.
\end{proof}

\section{Orlicz-Calder\'on Hardy spaces on the Heisenberg group} \label{OCH}

We recall that $\mathcal{P}_{k}$ is the subspace formed by all the polynomials of homogeneous degree at most $k \geq 0$. Now, the argument used to prove \cite[Lemma 2.6]{Zygm} works on $\mathbb{H}^n$ as well. Then, its analogous on $\mathbb{H}^n$, it is as follows.

\begin{lemma} \label{ident pol}
Given a integer $k \geq 0$, there exists $\varphi \in C^{\infty}(\mathbb{H}^n)$ with support on the $\rho$-ball $B(e,1)$ such that for every 
$\lambda > 0$ and every polynomial $p$ of homogeneous degree at most $k$,
\[
p(z) = \int \lambda^{Q} \varphi(\lambda \cdot (w^{-1} \cdot z)) \, p(w) dw =  (p \ast \varphi_{\lambda^{-1}})(z)
\]
holds.
\end{lemma}

\begin{remark} \label{der pol}
From this Lemma, one has that $X^{I}p = p \ast X^{I}(\varphi_{\lambda^{-1}})$ for every multi-index $I$.
\end{remark}

Let $L^{q}_{loc}(\mathbb{H}^{n})$, $1 < q < \infty$, be the space of all measurable functions $g$ on $\mathbb{H}^{n}$ that belong locally 
to $L^{q}$ for compact sets of $\mathbb{H}^{n}$. We endowed $L^{q}_{loc}(\mathbb{H}^{n})$ with the topology generated by the seminorms
\[
|g|_{q, \, B} = \left( |B|^{-1} \int_{B} \, |g(w)|^{q}\, dw \right)^{1/q},
\]
where $B$ is a $\rho$-ball in $\mathbb{H}^{n}$ and $|B|$ denotes its Haar measure.

For $g \in L^{q}_{loc}(\mathbb{H}^{n})$, we define a maximal function $\eta_{q, \, \gamma}(g; z)$ as
\[
\eta_{q, \, \gamma}(g; \, z) = \sup_{r > 0} r^{-\gamma} |g|_{q, \, B(z, r)},
\]
where $\gamma$ is a positive real number and $B(z, r)$ is the $\rho$-ball centered at $z$ with radius $r$.

We denote by $E^{q}_{k}$ the quotient space of $L^{q}_{loc}(\mathbb{H}^{n})$ by $\mathcal{P}_{k}$. If 
$G \in E^{q}_{k}$, we define the seminorm $\| G \|_{q, \, B} = \inf \left\{ |g|_{q, \, B} : g \in G \right\}$. The family of all these seminorms induces on $E^{q}_{k}$ the quotient topology.

Given a positive real number $\gamma$, we can write $\gamma = k + t$, where $k$ is a non negative integer and $0 < t \leq 1$. This decomposition is unique.

For $G \in E^{q}_{k}$, we define a maximal function $N_{q, \, \gamma}(G; z)$ as 
\[
N_{q, \, \gamma}(G; z) = \inf \left\{ \eta_{q, \, \gamma}(g; z) : g \in G \right\}.
\]

\begin{lemma} (\cite[Lemma 4.1]{rocha3}) \label{semicont} The maximal function $z \to N_{q; \, \gamma}(G; z)$ associated with a class $G$ in 
$E_{k}^{q}$ is lower semicontinuous.
\end{lemma}

\begin{definition}
Let $\Phi$ be an Orlicz function, we say that an element $G \in E^{q}_{k}$ belongs to the Orlicz-Calder\'on Hardy space
$\mathcal{H}^{\Phi}_{q, \, \gamma}(\mathbb{H}^{n})$ if the maximal function $N_{q, \, \gamma}(G; \, \cdot \,) \in L^{\Phi}(\mathbb{H}^{n})$. 
In this case, for any $G \in \mathcal{H}^{\Phi}_{q, \, \gamma}(\mathbb{H}^{n})$, we set
\[
\| G \|_{\mathcal{H}^{\Phi}_{q, \, \gamma}(\mathbb{H}^{n})} := \| N_{q, \, \gamma}(G; \, \cdot \,) \|_{\Phi}.
\]
\end{definition}

If $\Phi$ is an Orlicz function with positive lower type $p_{\Phi}^{-}$ and positive upper type $p_{\Phi}^{+}$, by 
Proposition \ref{quasi-norm}, it follows that the couple
$\left( \mathcal{H}^{\Phi}_{q, \, \gamma}(\mathbb{H}^{n}),  \| \cdot \|_{\mathcal{H}^{\Phi}_{q, \, \gamma}(\mathbb{H}^{n})} \right)$ results 
a quasi-normed space.

Now, taking into account Lemma \ref{ident pol} and Remark \ref{der pol}, to follow the proof of \cite[Lemma 3]{calderon}, it obtains the following result.

\begin{lemma} \label{estim pol}
Let $g_1$ and $g_2$ be two representatives of an element $G \in E^q_k$ and $p = g_1 - g_2$ ($p$ is a polynomial of homogeneous degree at 
most $k$). Then, for every multi-index $I$ with $0 \leq d(I) \leq k$, there exists a positive constant $C = C_{I}$ such that
\[
|X^{I}p(z)| \leq C \left[ \eta_{q, \gamma}(g_1; z_1) + \eta_{q, \gamma}(g_2; z_2) \right] 
\left( \rho(z_1^{-1} \cdot z) + \rho(z_2^{-1} \cdot z) \right)^{k+1 - d(I)}
\]
holds for every $z_1$, $z_2$ and $z$ in $\mathbb{H}^n$.
\end{lemma}

\begin{lemma} \label{puntual 1} Let $G \in E^{q}_{k}$ with $N_{q, \, \gamma}(G; z_0) < \infty,$ for some $z_0 \in \mathbb{H}^{n}$. Then:

\qquad

$(i)$ There exists a unique $g \in G$ such that $\eta_{q, \, \gamma} (g; z_0) < \infty$ and, therefore, 
$\eta_{q, \, \gamma} (g; z_0) = N_{q, \, \gamma}(G; z_0)$.

$(ii)$ For any $\rho$-ball $B$, there is a constant $c$ depending on $z_0$ and $B$ such that if $g$ is the unique representative of $G$ given in $(i)$, then
$$\|G\|_{q, \, B} \leq |g|_{q, \, B} \leq c \, \eta_{q, \, \gamma} (g; z_0) = c \, N_{q, \, \gamma}(G; z_0).$$

The constant $c$ can be chosen independently of $z_0$ provided that $z_0$ varies in a compact set.
\end{lemma}

\begin{proof}
To prove (i), we assume that $g_1$ and $g_2$ belong to $G$ and both $\eta_{q, \gamma}(g_1; z_0)$ and 
$\eta_{q, \gamma}(g_2; z_0)$ are finite. We call $p$ to the polynomial $g_1 - g_2$ of homogeneous degree at most $k$. Applying 
Lemma \ref{estim pol} with $z = z_1 = z_2 = z_0$ we have $X^{I}p(z_0) = 0$ for every multi-index $I$ such that $0 \leq d(I) \leq k$. Since every polynomial of homogeneous degree at most $k$ can be centered at $z_0$, with $z_0$ being an arbitrary point of 
$\mathbb{H}^n$ (see the formula that appears in \cite[Section 5.2, p. 272]{Bonfi} for the Taylor polynomial of a smooth function), we have that $p \equiv 0$. Then, (i) follows. Finally, (i) implies (ii).
\end{proof}

\begin{corollary} If $\{ G_{j} \}$ is a sequence of elements of $E^{q}_{k}$ converging to $G$ in 
$\mathcal{H}^{\Phi}_{q, \, \gamma}(\mathbb{H}^{n})$, then $\{ G_{j} \}$ converges to $G$ in $E^{q}_{k}$.
\end{corollary}

\begin{proof} For any $\rho$-ball $B$, by $(ii)$ of Lemma \ref{puntual 1}, we have
$$\| G- G_{j} \|_{q, \, B} \leq c \, \| \chi_{B} \|_{\Phi}^{-1} 
\| \chi_{B} \,\, N_{q, \, \gamma}(G - G_{j}; \, \cdot \,) \|_{\Phi} 
\leq c \, \| G - G_{j} \|_{\mathcal{H}^{\Phi}_{q, \, \gamma}(\mathbb{H}^{n})},$$
which proves the corollary.
\end{proof}

\begin{lemma} \label{series in Eqk} Let $\{ G_{j} \}$ be a sequence in $E^{q}_{k}$ such that for a given point $z_0 \in \mathbb{H}^n$, the series 
$\sum_j N_{q, \, \gamma}(G_{j}; \, z_0 )$ is finite. Then

\qquad

$(i)$ The series $\sum_j G_j$ converges in $E_{k}^{q}$ to an element $G$ and 
$$N_{q, \, \gamma}(G; \, z_0 ) \leq \sum_j N_{q, \, \gamma}(G_{j}; \, z_0 ).$$

$(ii)$ If $g_j$ is the unique representative of $G_j$ satisfying
$\eta_{q, \, \gamma} (g_j; z_0) = N_{q, \, \gamma}(G_j; z_0)$, then $\sum_j g_j$ converges in $L^{q}_{loc}(\mathbb{H}^{n})$ to a function 
$g$ that is the unique representative of $G$ satisfying $\eta_{q, \, \gamma} (g; z_0) = N_{q, \, \gamma}(G; z_0)$
\end{lemma}

\begin{proof} The proof is similar to the one given in \cite[Lemma 4]{segovia}.
\end{proof}

\begin{proposition} (\cite[Proposition 4.7]{rocha3}) \label{g distrib}
If $g \in L^{q}_{loc}(\mathbb{H}^{n})$, $1 < q < \infty$, and there is a point $z_0 \in \mathbb{H}^{n}$ such that 
$\eta_{q, \, \gamma} (g ; z_0) < \infty$, then $g \in \mathcal{S}'(\mathbb{H}^{n})$.
\end{proposition}

\begin{proposition} \label{Lg dist} Let $g \in L^q_{loc} \cap \mathcal{S}'(\mathbb{H}^n)$ and 
$f = \mathcal{L} g$ in $\mathcal{S}'(\mathbb{H}^n)$. If $\phi \in \mathcal{S}(\mathbb{H}^n)$ and $N > Q+2$, then
\[
(M_{\phi}^{dis}f)(z) := \sup \left\{ |(f \ast \phi_{2^{j}})(z)| : j \in \mathbb{Z} \right\}  
\]
\[
\leq C \| \phi \|_{\mathcal{S}(\mathbb{H}^{n}), N}  \,\,\, \eta_{q, 2}(g; \, z)
\]
holds for all $z \in \mathbb{H}^n$.
\end{proposition}

\begin{proof} 
Since $(M_{\phi}^{dis}f)(z) \leq (M_{\phi} f)(z) := \sup \left\{ |(f \ast \phi_t)(w)| : \rho(w^{-1} \cdot z) < t, \, 0 < t < \infty \right\}$ for all $z \in \mathbb{H}^n$, the proposition follows from \cite[Proposition 4.8]{rocha3}.
\end{proof}

\begin{proposition} \label{cerrado} If $\Phi$ is an Orlicz function  with $0 < i(\Phi) \leq I(\Phi) < \infty$, then the space 
$\mathcal{H}^{\Phi}_{q, \, \gamma}(\mathbb{H}^{n})$ is complete.
\end{proposition}

\begin{proof} Without loss of generality, we may assume that $\Phi$ is continuous and strictly increasing (see Remark \ref{cont and crec}). 
By \cite[Theoremm 3.3]{Nekvinda}, it is enough to show that $\mathcal{H}^{\Phi}_{q, \, \gamma}$ has the generalized Riesz-Fisher property 
with constant $C_0 \in [1, \infty)$, i.e.: for any sequence $\{ G_j \}$ in $\mathcal{H}^{\Phi}_{q, \, \gamma}$ that satisfies 
\[
\sum_{j=1}^{\infty} C_{0}^{j+1} \| G_j \|_{\mathcal{H}^{\Phi}_{q, \, \gamma}} < \infty,
\]
the series $\sum_{j} G_j$ converges in $\mathcal{H}^{\Phi}_{q, \, \gamma}$.

Let $C_0 = K$, where $K$ is as in Proposition \ref{quasi-norm} - (iii), and let $m \geq 1$ be fixed, then 
\begin{eqnarray*}
\left\| \sum_{j=m}^{k} N_{q, \, \gamma}(G_{j}; \, \cdot \,) \right\|_{\Phi} &\leq& \sum_{j=m}^{k} C_{0}^{j-m+1} 
\left\| N_{q, \, \gamma}(G_{j}; \, \cdot \,) \right\|_{\Phi} \\
&\leq& \sum_{j=m}^{\infty} C_{0}^{j-m+1} \| G_j \|_{\mathcal{H}^{\Phi}_{q, \, \gamma}} =: \alpha_m < \infty,
\end{eqnarray*}
for every $k \geq m$. Being $\Phi$ an increasing continuous function, by Lemma \ref{fi menor uno}, we obtain 
\[
\int_{\mathbb{H}^{n}} \, \Phi\left( \alpha_m^{-1} \, \left|\sum_{j=m}^{k} N_{q, \, \gamma}(G_{j}; \, z ) \right| \right) \, dz
\] 
\[
\leq \int_{\mathbb{H}^{n}} \Phi \left( \left\| \sum_{j=m}^{k} N_{q, \, \gamma}(G_{j}; \, \cdot \,) \right\|_{\Phi} \, 
\left| \sum_{j=m}^{k} N_{q, \, \gamma}(G_{j};  z ) \right| \right) \, dz \leq 1, \,\,\, \forall \, k \geq m,
\]
by applying Fatou's lemma as $k \rightarrow \infty$ and the continuity of $\Phi$, we obtain
\[
\int_{\mathbb{H}^{n}} \, \Phi \left( \alpha_m^{-1} \, \left| \sum_{j=m}^{\infty} N_{q, \, \gamma}(G_{j}; \, z ) \right| \right) \, dz \leq 1,
\]
so
\begin{equation}
\left\|  \sum_{j=m}^{\infty} N_{q, \, \gamma}(G_{j}; \, \cdot \,) \right\|_{\Phi} \leq \alpha_m = \sum_{j=m}^{\infty} C_{0}^{j-m+1} 
\| G_j \|_{\mathcal{H}^{\Phi}_{q, \, \gamma}} < \infty, \,\,\,\, \forall \, m \geq 1  \label{serie}.
\end{equation}
Taking $m = 1$ in (\ref{serie}), it follows that $\sum_{j} N_{q, \, \gamma}(G_{j}; z)$ is finite a.e. 
$z \in \mathbb{H}^{n}$. Then, by $(i)$ of Lemma \ref{series in Eqk}, the series $\sum_j G_j$ converges in $E_{k}^{q}$ to an element $G$. Now 
\[
N_{q, \, \gamma}\left( G - \sum_{j=1}^{k} G_j; \, z \right) \leq \sum_{j=k+1}^{\infty} N_{q, \, \gamma} (G_j; \, z),
\]
from this and (\ref{serie}) we get
\[
\left\| G - \sum_{j=1}^{k} G_j \right\|_{\mathcal{H}^{\Phi}_{q, \, \gamma}} \leq \sum_{j=k+1}^{\infty} C_{0}^{j-k} 
\| G_j \|_{\mathcal{H}^{\Phi}_{q, \, \gamma}},
\]
and since the right-hand side tends to $0$ as $k \rightarrow \infty$, the series $\sum_{j}G_j$ converges to $G$ in $\mathcal{H}^{\Phi}_{q, \, \gamma}(\mathbb{H}^{n})$.
\end{proof}

\begin{remark} \label{LG definition} We observe that if $G \in \mathcal{H}^{\Phi}_{q, \, 2}(\mathbb{H}^{n})$, then 
$N_{q, \, 2}(G; z_0) < \infty,$ for some $z_0 \in \mathbb{H}^{n}$. By $(i)$ in Lemma \ref{puntual 1} there exists $g \in G$ such that 
$N_{q, \, 2}(G; z_0) = \eta_{q, \, 2}(g; z_0)$; from Proposition \ref{g distrib} it follows that $g  \in \mathcal{S}'(\mathbb{H}^{n})$. So 
$\mathcal{L} g$ is well defined in sense of distributions. On the other hand, since any two representatives of $G$ differ in a polynomial 
of homogeneous degree at most $1$, we get that $\mathcal{L} g$ is independent of the representative $g \in G$ chosen. Therefore, for 
$G \in \mathcal{H}^{\Phi}_{q, \, 2}(\mathbb{H}^{n})$, we define $\mathcal{L} G$ as the distribution $\mathcal{L} g$, where $g$ is any representative of $G$.
\end{remark}

\begin{theorem} \label{sub-lap inject}
Let $\Phi$ be an Orlicz function of positive upper type $p_{\Phi}^{+}$. If $G \in \mathcal{H}^{\Phi}_{q, 2}(\mathbb{H}^{n})$ and 
$\mathcal{L} G = 0$, then $G \equiv 0$.
\end{theorem}

\begin{proof}
Given $G \in \mathcal{H}^{\Phi}_{q, 2}(\mathbb{H}^{n})$, let $\mathcal{O} := \{ z \in \mathbb{H}^n : N_{q, 2}(G; z) > 1  \}$. By Lemma 
\ref{semicont}, the set $\mathcal{O}$ is open with $| \mathcal{O} | < \infty$, and by Lemmas \ref{lower estim} and \ref{modular 1}, we obtain
\begin{eqnarray*}
\int_{\mathbb{H}^n \setminus \mathcal{O}} N_{q, 2}(G; z)^{p_{\Phi}^{+}} dz &\lesssim& \int_{\mathbb{H}^n \setminus \mathcal{O}} 
\Phi \left( N_{q, 2}(G; z) \right) dz \\ &\leq& \int_{\mathbb{H}^n} \Phi \left( N_{q, 2}(G; z) \right) dz < \infty.
\end{eqnarray*}
Then, following the proof of \cite[Theorem 4.10]{rocha3} and applying \cite[Lemma 2.5]{rocha3} with 
$h(\cdot) = N_{q, 2}(G; \cdot) \in L^{p_{\Phi}^{+}}(\mathbb{H}^n \setminus \mathcal{O})$ together with Remark \ref{LG definition}, the 
theorem follows.
\end{proof}

If $a$ is a bounded function with compact support, its potential $b$, defined as 
\begin{equation} \label{potencial b}
b(z) := \left( a \ast c_n \, \rho^{-2n} \right)(z) = c_n \int_{\mathbb{H}^{n}} \rho(w^{-1} \cdot z)^{-2n} a(w) dw,
\end{equation}
is a locally bounded function and, by Theorem \ref{fund solution}, $\mathcal{L} b = a$ in the sense of distributions.

\

In the sequel, $Q = 2n+2$ and $\beta$ is the constant in \cite[Corollary 1.44]{Foll-St}, we observe that $\beta \geq 1$ 
(see \cite[p. 29]{Foll-St}). For the potentials (\ref{potencial b}), we have the following result.

\begin{lemma} \label{a conv ro} Let $a(\cdot)$ be an $(\Phi, p_{0}, m)$ - atom centered at the $\rho$ - ball $B(z_0, \delta)$. If 
\[
b(z) = \left( a \ast c_n \, \rho^{-2n} \right)(z),
\]
then, for $\rho(z_0^{-1} z) \geq 2 \beta^{2}\delta$ and every multi-index $I$ there exists a positive constant $C_{I}$ such that 
\[
\left| (X^{I}b)(z) \right| \leq C_{I} \, \delta^{2+Q} \| \chi_B \|^{-1}_{\Phi} \rho(z_{0}^{-1} \cdot z)^{-Q-d(I)}
\]
holds.
\end{lemma}

\begin{proof} 
The proof is similar to the one given in \cite[Lemma 4.11]{rocha3}, but considering now Definition \ref{atom def}. 
\end{proof}

The following result is crucial to get our main result.

\begin{proposition} \label{pointwise estimate} Let $a(\cdot)$ be an $(\Phi, p_{0}, m)$ - atom centered at the $\rho$ - ball 
$B=B(z_0, \delta)$. If $b(z) = (a \ast c_n \rho^{-2n})(z)$, then for all $z \in \mathbb{H}^{n}$
\begin{eqnarray} 
\label{N estimate}
N_{q, 2} \left(\widetilde{b}; \, z \right) &\lesssim& \| \chi_B \|^{-1}_{\Phi} \left[(M \chi_{B})(z) \right]^{\frac{2 + Q/q}{Q}} + 
\chi_{4 \beta^2 B}(z) (M a)(z)  \\
\notag
&+& \chi_{4 \beta^2 B}(z) \sum_{d(I)=2} (T^{*}_{I} a)(z),
\end{eqnarray}
where $\widetilde{b}$ is the class of $b$ in $E^{q}_{1}$, $M$ is the Hardy-Littlewood maximal operator and $(T^{*}_{I} a) (z) = \sup_{\epsilon >0} \left|\int_{\rho(w^{-1} \cdot z) > \epsilon} \, (X^{I} \rho^{-2n})(w^{-1} \cdot z) a(w) \, dw \right|$.
\end{proposition}

\begin{proof} We point out that the argument used in the proof of \cite[Proposition 4.12]{rocha3}, to obtain the pointwise inequality
(4.9) therein, works in this setting as well, but considering now the conditions $a1)$, $a2)$ and $a3)$ given in Definition 
\ref{atom def} of $(\Phi, p_0, m)$ - atom. These conditions are similar to those of the atoms in classical context (see p. 71-72 
in \cite{Foll-St}). Then, this observation and Lemma \ref{a conv ro} allow us to get (\ref{N estimate}).
\end{proof}

\section{Main results} \label{main thm}

We are now in a position to prove our main results.

\begin{theorem} \label{principal thm} Let $Q=2n+2$, $1 < q < \frac{n+1}{n}$ and let $\Phi$ be an Orlicz function such that 
$Q \, (2 + \frac{Q}{q})^{-1} < i(\Phi) \leq I(\Phi) < \infty$. Then the Heisenberg 
sub-laplacian $\mathcal{L}$ is a bijective mapping from $\mathcal{H}^{\Phi}_{q, 2}(\mathbb{H}^{n})$ onto $H^{\Phi}(\mathbb{H}^{n})$. 
Moreover, there exist two positive constant $c_1$ and $c_2$ such that
\begin{equation} \label{doble ineq}
c_1 \|G \|_{\mathcal{H}^{\Phi}_{q, 2}(\mathbb{H}^{n})} \leq \| \mathcal{L}G \|_{H^{\Phi}(\mathbb{H}^{n})} \leq 
c_2 \|G \|_{\mathcal{H}^{\Phi}_{q, 2}(\mathbb{H}^{n})}
\end{equation}
hold for all $G \in \mathcal{H}^{\Phi}_{q, 2}(\mathbb{H}^{n})$.
\end{theorem}

\begin{proof} The injectivity of the sub-laplacion $\mathcal{L}$ in $\mathcal{H}^{\Phi}_{q, 2}(\mathbb{H}^{n})$ was proved in 
Theorem \ref{sub-lap inject}. Now, let $G \in \mathcal{H}^{\Phi}_{q, \, 2}(\mathbb{H}^{n})$, since $N_{q, 2}(G; z)$ is finite 
$\text{a.e.} \,\, z \in \mathbb{H}^{n}$, by $(i)$ in Lemma \ref{puntual 1} and Proposition \ref{g distrib} the unique 
representative $g$ of $G$ (which depends on $z$), satisfying $\eta_{q, 2}(g; z) = N_{q, 2}(G; z)$, is a function in 
$L^{q}_{loc}(\mathbb{H}^{n}) \cap \mathcal{S}'(\mathbb{H}^{n})$. In particular, for a radial function $\phi \in \mathcal{S}(\mathbb{H}^n)$ with $\int \phi = 1$, by Remark \ref{LG definition} and Proposition \ref{Lg dist} we get
\[
M_{\phi}^{dis}(\mathcal{L}G)(z) \leq C \| \phi \|_{\mathcal{S}(\mathbb{H}^{n}), N}  \,\,\, N_{q, \, 2}(G; z).
\]
Then, this inequality and Theorem \ref{norm comparable} give $\mathcal{L}G \in H^{\Phi}(\mathbb{H}^{n})$ and 
\begin{equation} \label{continuity}
\| \mathcal{L}G \|_{H^{\Phi}(\mathbb{H}^{n})} \leq C \, \| G \|_{\mathcal{H}^{\Phi}_{q, \, 2}(\mathbb{H}^{n})}.
\end{equation}
This proves the continuity of sublaplacian $\mathcal{L}$ from $\mathcal{H}^{\Phi}_{q, \, 2}(\mathbb{H}^{n})$ into $H^{\Phi}(\mathbb{H}^{n})$.

Next, we shall see that the operator $\mathcal{L}$ is onto. Given $f \in H^{\Phi}(\mathbb{H}^{n})$, by Theorem \ref{infinite atomic decomp}, there exist a sequence of nonnegative numbers $\{ \lambda_j \}_{j=1}^{\infty}$ and a sequence of $\rho$ - balls $\{B_j \}_{j=1}^{\infty}$ and 
$(\Phi, \infty, m)$ - atoms $a_j$ supported on $B_j$, such that $f= \sum_{j=1}^{\infty} \lambda_j a_j$ and
\begin{equation} \label{atomic ineq}
\left\Vert \left\{ \sum_{j} \left( \frac{\lambda_j \chi_{B_j}}{\Vert \chi_{B_j} \Vert_{\Phi}} \right)^{\theta} \right\}^{1/\theta} 
\right\Vert_{\Phi} \lesssim \|f \|_{H^{\Phi}(\mathbb{H}^{n})}, \,\,\,\,\,\,\,\, \forall \,\,\, 0 < \theta < \min \{ 1, i(\Phi) \}.
\end{equation}
For each $j \in \mathbb{N}$ we put $b_j(z)= (a_j \ast c_n \rho^{-2n})(z) = \int_{\mathbb{H}^{n}} c_n \rho(w^{-1} \cdot z)^{-2n} a_j(w) dw$, from Proposition \ref{pointwise estimate} we have
\[
N_{q, 2} \left(\widetilde{b}_j; \, z \right) \lesssim \| \chi_B \|^{-1}_{\Phi} \left[(M \chi_{B_j})(z) \right]^{\frac{2 + Q/q}{Q}} + 
\chi_{4 \beta^2 B_j}(z) (M a_j)(z) 
\]
\[
+ \chi_{4 \beta^2 B_j}(z) \sum_{d(I)=2} (T^{*}_{I} a_j)(z),
\]
so
\begin{eqnarray*}
\sum_{j=1}^{\infty} \lambda_j N_{q, 2} \left(\widetilde{b}_j; \, z \right) &\lesssim& \sum_{j=1}^{\infty} \lambda_j \| \chi_B \|^{-1}_{\Phi} 
\left[(M \chi_{B_j})(z) \right]^{\frac{2 + Q/q}{Q}} \\
&+& \sum_{j=1}^{\infty} \lambda_j \chi_{4 \beta^2 B_j}(z) (M a_j)(z) \\
&+& \sum_{j=1}^{\infty} \lambda_j \chi_{4 \beta^2 B_j}(z) \sum_{d(I)=2} (T^{*}_{I} a_j)(z) \\
&=:& J_1 + J_2 + J_3.
\end{eqnarray*}
To study $J_1$, by hypothesis, we have that $i(\Phi) > Q \, (2 + \frac{Q}{q})^{-1}$. Then, by Lemma \ref{fits},
\begin{eqnarray*}
\|J_1\|_{\Phi} & = & \left\| \sum_{j=1}^{\infty} \frac{\lambda_j}{\| \chi_B \|_{\Phi}} M(\chi_{B_j})(\cdot)^{\frac{2 + Q/q}{Q}} 
\right\|_{\Phi} \\
& = & \left\| \left\{ \sum_{j=1}^{\infty} \frac{\lambda_j}{\| \chi_B \|_{\Phi}} M(\chi_{B_j})(\cdot)^{\frac{2 + Q/q}{Q}} 
\right\}^{\frac{Q}{2 + Q/q}} \right\|_{\Phi_{\frac{2 + Q/q}{Q}}}^{\frac{2 + Q/q}{Q}}
\end{eqnarray*}
\begin{eqnarray*}
& \lesssim & \left\| \left\{ \sum_{j=1}^{\infty} \frac{\lambda_j}{\| \chi_B \|_{\Phi}} \chi_{B_j}  \right\}^{\frac{Q}{2 + Q/q}} 
\right\|_{\Phi_{\frac{2 + Q/q}{Q}}}^{\frac{2 + Q/q}{Q}} \\
& = & \left\| \sum_{j=1}^{\infty} \frac{\lambda_j}{\| \chi_B \|_{\Phi}} \chi_{B_j} \right\|_{\Phi} \\
& \leq & \left\| \left\{ \sum_{j=1}^{\infty} \left( \frac{\lambda_j}{\| \chi_B \|_{\Phi}} \chi_{B_j} \right)^{\theta} 
\right\}^{1/\theta}\right\|_{\Phi} \lesssim \|f \|_{H^{\Phi}(\mathbb{H}^{n})},
\end{eqnarray*}
where the first inequality follows from Lemma \ref{lower upper fits} and Theorem \ref{Feff-Stein max ineq}, 
since $i(\Phi) > Q \, (2 + \frac{Q}{q})^{-1}$ and $(2 + \frac{Q}{q})/Q > 1$, the embedding 
$\ell^{\theta}(\mathbb{N}) \hookrightarrow \ell^{1}(\mathbb{N})$ gives the second inequality, 
and (\ref{atomic ineq}) gives the last one.

To estimate $J_2$, we apply Proposition \ref{max on atoms} with $r = 4 \beta^2 \geq 1$, Remark \ref{atomo infinito} and 
(\ref{atomic ineq}). Then, we obtain
\[
\| J_2 \|_{\Phi} \lesssim \|f \|_{H^{\Phi}(\mathbb{H}^{n})}.
\]

To study $J_3$, by Theorem 3 in \cite{Folland} and Corollary 2, p. 36, in \cite{Elias} (see also {\bf 2.5}, p. 11, in \cite{Elias}), 
we have, for every multi-index $I$ with $d(I) = 2$, that the operator $T_{I}^{*}$ is bounded on $L^{p_0}(\mathbb{H}^n)$ for each 
$1 < p_0 < \infty$. So, Proposition \ref{max on atoms} holds with the operators $\{ T_{I}^{*} \}_{d(I) = 2}$ instead of $M$ and with the $(\Phi, \infty, m)$ - atoms $a_j$ (see Remark \ref{atomo infinito}). Then, by 
(\ref{atomic ineq}), we get
\[
\| J_3 \|_{\Phi} \lesssim \|f \|_{H^{\Phi}(\mathbb{H}^{n})}.
\]

Thus,
\[
\left\| \sum_{j=1}^{\infty} \lambda_j N_{q, 2} \left( \widetilde{b}_j; \, \cdot \right) \right\|_{\Phi} \lesssim 
\|f \|_{{H^{\Phi}(\mathbb{H}^n)}}.
\]
By Lemma \ref{modular 1}, $\kappa_{\Phi}\left( \sum_{j=1}^{\infty} \lambda_j N_{q, 2}(\widetilde{b}_j; \cdot) \right) < \infty$. Hence
\begin{equation}
\sum_{j=1}^{\infty} \lambda_j N_{q, 2}(\widetilde{b}_j; \, z) < \infty \,\,\,\,\,\, \text{a.e.} \, z \in \mathbb{H}^{n} \label{Nq}
\end{equation}
and
\begin{equation}
\kappa_{\Phi}\left( \sum_{j=M+1}^{\infty} \lambda_j N_{q, 2}(\widetilde{b}_j; \, \cdot) \right) \rightarrow 0, \,\,\,\, \text{as} \,\,  M \rightarrow \infty  \label{Nq2}.
\end{equation}
From (\ref{Nq}) and Lemma \ref{series in Eqk}, there exists a function $G$ such that $\sum_{j=1}^{\infty} \lambda_j \widetilde{b}_j = G$ 
in $E^{q}_{1}$ and
\[
N_{q, 2} \left( \left(G - \sum_{j=1}^{M} \lambda_j \widetilde{b}_j \right) ; \, z \right) \leq c \, 
\sum_{j=M+1}^{\infty} \lambda_j N_{q, 2}(\widetilde{b}_j; z).
\]
This estimate together with (\ref{Nq2}) and Lemma \ref{modular 2} implies
\[
\left\| G - \sum_{j=1}^{M} \lambda_j \widetilde{b}_j \right\|_{\mathcal{H}^{\Phi}_{q,2}} \rightarrow 0, \,\,\,\, \text{as} \,\,  M \rightarrow \infty.
\]
So $G \in \mathcal{H}^{\Phi}_{q,2}(\mathbb{H}^{n})$ and, by Proposition \ref{cerrado}, $G = \sum_{j=1}^{\infty} \lambda_j \widetilde{b}_j$ 
in $\mathcal{H}^{\Phi}_{q,2}(\mathbb{H}^{n})$. Since $\mathcal{L}$ is a continuous operator from $\mathcal{H}^{\Phi}_{q,2}(\mathbb{H}^{n})$ into $H^{\Phi}(\mathbb{H}^{n})$, we get
\[
\mathcal{L}G = \sum_j \lambda_j \mathcal{L} \widetilde{b}_j = \sum_j \lambda_j a_j = f,
\]
in $H^{\Phi}(\mathbb{H}^{n})$. This shows that $\mathcal{L}$ is onto $H^{\Phi}(\mathbb{H}^{n})$. Moreover,
\begin{equation} \label{continuity 2}
\| G\|_{\mathcal{H}^{\Phi}_{q,2}} = \left\| \sum_{j=1}^{\infty} \lambda_j \widetilde{b}_j \right\|_{\mathcal{H}^{\Phi}_{q,2}} \lesssim 
\left\| \sum_{j=1}^{\infty} \lambda_j N_{q, 2}(\widetilde{b}_j; \cdot) \right\|_{\Phi} \lesssim \|f \|_{H^{\Phi}} = 
\| \mathcal{L} G \|_{H^{\Phi}}.
\end{equation}
Finally, (\ref{continuity}) and (\ref{continuity 2}) give (\ref{doble ineq}), and so the proof is concluded.
\end{proof}

We shall now see that the case $0 < I(\Phi) < Q \, (2 + \frac{Q}{q})^{-1}$ is trivial.

\begin{theorem} \label{2nd thm}
If $1 < q < \frac{n+1}{n}$ and $\Phi$ is an Orlicz function with $0 < I(\Phi) < Q \, (2 + \frac{Q}{q})^{-1}$, then 
$\mathcal{H}^{\Phi}_{q, \, 2}(\mathbb{H}^{n}) = \{ 0 \}.$
\end{theorem}

\begin{proof} Let $G \in \mathcal{H}^{\Phi}_{q, \, 2}(\mathbb{H}^{n})$ and assume $G \neq 0$. Then there exists $g \in G$ 
that is not a polynomial of homogeneous degree less or equal to $1$. It is easy to check that there exist a positive constant 
$c \in (0, 1]$ and a $\rho$ - ball $B = B(e, r)$ with $r > 1$ such that
\[
\int_{B} |g(w) - P(w)|^{q} \, dw \geq c > 0,
\]
for every $P \in \mathcal{P}_{1}$.

Let $z$ be a point such that $\rho(z) > r$ and let $\delta = 2 \rho(z)$. Then $B(e, r) \subset B(z, \delta)$. If $h \in G$, then 
$h = g - P$ for some $P \in \mathcal{P}_{1}$ and
\[
\delta^{-2}|h|_{q, B(z, \delta)} \geq c \rho(z)^{-2-Q/q}.
\]
So $N_{q,2}(G; \, z) \geq c \, \rho(z)^{-(2+Q/q)}$, for $\rho(z) > r$. Since $\Phi$ is an Orlicz function of positive upper type 
$p_{\Phi}^{+}$ with $I(\Phi) \leq p_{\Phi}^{+} < Q(2+Q/q)^{-1}$ and $0 < c \, \rho(z)^{-(2+Q/q)} < 1$ on $\rho(z) > r$, by 
Lemma \ref{lower estim} - (ii), we have that
\begin{eqnarray*}
\int_{\mathbb{H}^n} \Phi\left( N_{q,2}(G; z) \right) dz &\geq& \, \int_{\rho(z) > r} \Phi\left( c \rho(z)^{-(2+Q/q)} \right) \, dz \\
&\gtrsim& \int_{\rho(z) > r} \rho(z)^{-(2+Q/q)p^{+}_{\Phi}} \, dz = \infty,
\end{eqnarray*}
which gives a contradiction (see Lemma \ref{modular 1}). Thus $\mathcal{H}^{\Phi}_{q, 2}(\mathbb{H}^{n}) = \{0\}$, 
if $0 < I(\Phi) < Q(2+Q/q)^{-1}$.
\end{proof}

Pablo Rocha, Instituto de Matem\'atica (INMABB), Departamento de Matem\'atica, Universidad Nacional del Sur (UNS)-CONICET, Bah\'ia Blanca, Argentina. \\
{\it e-mail:} pablo.rocha@uns.edu.ar

\end{document}